\documentclass[]{siamart190516}

\usepackage[T1]{fontenc}
\usepackage[utf8]{inputenc}
\usepackage[english]{babel}
\usepackage{mleftright}
\usepackage[inline,shortlabels]{enumitem}
\usepackage{graphicx}
\usepackage{subfig}
\usepackage{amsmath,amsfonts,amssymb}
\usepackage[noadjust]{cite}
\usepackage{xcolor}
\usepackage{verbatim}

\allowdisplaybreaks

\usepackage{tikz}
\usetikzlibrary{arrows}

\tikzset{
    punkt/.style={
           rectangle,
           draw=white, very thick,
           text width=6.5em,
           minimum height=1.5em,
           text centered}
}

\usepackage{etoolbox}
\patchcmd{\SetTagPlusEndMark}{$}{}{}{}
\patchcmd{\SetTagPlusEndMark}{$}{}{}{}

\setlist{font=\normalfont,topsep=1ex,parsep=0ex}

\newcommand{\ConX}{C}
\newcommand{\SpaceY}{Y}
\newcommand{\ConY}{K}
\newcommand{\SpaceH}{H}
\newcommand{\ConH}{\mathcal{K}}
\newcommand{\Feas}{\mathcal{F}}
\newcommand{\Lag}{L}

\newcommand{\BddMul}{w}

\newcommand{\R}{\mathbb{R}}
\newcommand{\M}{\mathcal{M}}
\newcommand{\MM}{\M}

\newcommand{\N}{\mathbb{N}}
\newcommand{\Nor}{\mathcal{N}}
\newcommand\NN\Nor
\newcommand{\LL}{\mathcal{L}}
\newcommand{\Tng}{\mathcal{T}}
\newcommand\TT\Tng
\newcommand{\Rad}{\mathcal{R}}
\newcommand\RR\Rad
\newcommand{\Lin}{\mathcal{L}}
\newcommand{\wto}{\rightharpoonup}
\newcommand{\weakly}{\wto}
\newcommand{\wtos}{\wto^*}

\newcommand{\wStarlimsup}{\operatorname{w*-}\limsup}
\DeclareMathOperator{\dist}{dist}
\newcommand{\<}{\langle}
\renewcommand{\>}{\rangle}
\newcommand{\minimize}{\operatornamewithlimits{minimize}}
\newcommand{\subjectto}{\textnormal{subject to}}
\newcommand{\dualspace}{^*}
\newcommand{\adjoint}{^*}

\newcommand{\cl}{\operatorname{cl}}
\newcommand{\polar}{^\circ}

\newcommand\range{\operatorname{ran}}
\newcommand\coker{\operatorname{coker}}

\newcommand{\vertiii}[1]{{\left\vert\kern-0.25ex\left\vert\kern-0.25ex\left\vert #1 
    \right\vert\kern-0.25ex\right\vert\kern-0.25ex\right\vert}}

\newsiamthm{myalgorithm}{Algorithm}
\newsiamthm{assumption}{Assumption}
\newsiamthm{example}{Example}
\newsiamthm{remark}{Remark}

\numberwithin{equation}{section}

\DeclareMathAlphabet{\mathpzc}{OT1}{pzc}{m}{it}
\newcommand\oo{\mathpzc{o}}

\usepackage{mathtools}
\DeclarePairedDelimiter\abs{\lvert}{\rvert}
\DeclarePairedDelimiter\norm{\lVert}{\rVert}
\DeclarePairedDelimiterX\innerp[2]{(}{)}{#1,#2}
\DeclarePairedDelimiterX\dual[2]{\langle}{\rangle}{#1,#2}
\providecommand\given{\nonscript\;\delimsize|\nonscript\;}
\DeclarePairedDelimiterX\set[1]{\{}{\}}{#1}

\DeclarePairedDelimiter\paren{(}{)}

\begin{document}

\title{
	New Constraint Qualifications for Optimization Problems in Banach Spaces 
	based on Asymptotic KKT Conditions
}

\author{Eike Börgens\footnotemark[1]
	\and
	Christian Kanzow\thanks{%
		University of Würzburg,
		Institute of Mathematics,
		97074 Würzburg,
		Germany,
		email: \{eike.boergens,kanzow\}@mathematik.uni-wuerzburg.de,
	}
	\and
	Patrick Mehlitz\footnotemark[2]
	\and 
	Gerd Wachsmuth\thanks{%
		Brandenburgische Technische Universität Cottbus--Senftenberg,
		Institute of Mathematics,
		03046 Cottbus,
		Germany,
		email: \{mehlitz,wachsmuth\}@b-tu.de.
		\funding{This research was 
			supported by the German Research Foundation (DFG) within the priority 
			program ``Non-smooth and Complementarity-based Distributed Parameter Systems: 
			Simulation and Hierarchical Optimization'' 
			(SPP 1962) under grant numbers KA 1296/24-2 and WA 3636/4-2.
		}
	} 
}

\date{\today}
\maketitle

\begin{abstract}
Optimization theory in Banach spaces suffers from the lack of available constraint qualifications.
Despite the fact that there exist only a very few constraint qualifications, they are, in addition,
often violated even in simple applications. This is very much in contrast to finite-dimensional
nonlinear programs, where a large number of constraint qualifications is known. Since these
constraint qualifications are usually defined using the set of active inequality constraints, it
is difficult to extend them to the infinite-dimensional setting. One exception is a recently
introduced sequential constraint qualification based on 
\emph{asymptotic} KKT conditions. 
This paper shows that this so-called
asymptotic KKT regularity 
allows suitable extensions to the Banach space setting
in order to obtain new constraint qualifications.
The relation of these new 
constraint qualifications to existing ones is discussed in detail.
Their usefulness is also shown
by several examples as well as an algorithmic application to the class of augmented Lagrangian methods.
\end{abstract}



\section{Introduction}

We consider the \emph{Banach space optimization problem}
\begin{equation}\label{Eq:Opt}\tag{$P$}
    \begin{aligned}
    	&\minimize_{x\in\ConX}&  &f(x)&
    	&\subjectto& &G(x)\in\ConY,&
    \end{aligned}
\end{equation}
where $X$ and $\SpaceY$ are (real) Banach spaces, $f\colon X\to\R$ and $G\colon X\to \SpaceY$ are continuously 
Fréchet differentiable mappings, and
$\ConX\subset X$ as well as $\ConY\subset\SpaceY$ are nonempty, closed, convex sets. 
The feasible set of \eqref{Eq:Opt} will be denoted by $\Feas$, i.e., we use
\[
	\Feas:=\{x\in C\,|\,G(x)\in K\}.
\] 
In many cases, $ \ConY $ is actually a cone. 
Moreover, the abstract constraints represented by the set $ \ConX $ may not be present,
i.e., $ \ConX = X $ is possible. Problems akin to \eqref{Eq:Opt} have long been identified as a suitable framework 
for generic optimization covering models from standard nonlinear programming, conic programming, inverse
optimization, or optimal control, see e.g.\ \cite{Bonnans2000} for more details.

A central role for both the theoretical investigation and the numerical solution of optimization problems
like \eqref{Eq:Opt} is played by \emph{constraint qualifications} (CQs). 
The validity of such constraint qualifications at a local minimizer of \eqref{Eq:Opt}
implies that the so-called Karush--Kuhn--Tucker (KKT) conditions associated with this program hold
at the latter point.
Unfortunately, there is a major gap between finite- and infinite-dimensional optimization problems
regarding available CQs.

Some of the weaker CQs like the Abadie or the Guignard constraint
qualification, see \cite{Abadie1965} and \cite{Guignard1969}, respectively, 
can easily be defined for optimization problems in Banach spaces as well, but, as in
the finite-dimensional setting, these conditions are quite abstract and, thus, difficult to
check in practice.
Moreover, they are usually too weak in order to yield meaningful
consequences for convergence theory associated with optimization algorithms which can be used to 
tackle \eqref{Eq:Opt}. 
There exist only a very few
stronger constraint qualifications, but they are often not 
satisfied in practical applications. The most prominent example is probably \emph{Robinson's
constraint qualification} (RCQ), see \cite{Kurcyusz1976,Robinson1976,Zowe1979}, 
which is already violated in the very simple situation where
two-sided (pointwise) box constraints in Lebesgue spaces are under consideration. 
On the other hand, there exist numerous constraint qualifications
in the finite-dimensional context, but these typically depend on the notion of active constraints
and, therefore, cannot be translated directly to the infinite-dimensional setting. Here, RCQ is an
exception since it boils down to the well-known Mangasarian--Fromovitz constraint qualification
in finite dimensions.

Our aim is therefore to introduce new constraint qualifications for optimization problems in 
Banach spaces which are weaker than RCQ and, consequently, have a chance to be satisfied for a
significantly larger class of problems. These new constraint qualifications are so-called
sequential constraint qualifications and, thus, closely related to the notion of the 
asymptotic KKT conditions (AKKT). 
Our study is motivated by a recent
series of papers (dealing with finite-dimensional standard nonlinear programs) 
on the so-called AKKT regularity or cone continuity property, see, e.g.,
\cite{Andreani2010,Andreani2011,AndreaniMartinezRamosSilva2016,Andreani2019}. 
This paper is based on the seemingly simple, but important observation that this
AKKT regularity
can be formulated without using the notion of active constraints, and therefore allows an extension
to optimization problems in Banach spaces. 
Nevertheless, the generalization of the 
AKKT regularity
to Banach spaces requires some care due to the difference between weak and strong convergence.
On the other hand, the opportunity of distinguishing between strong and weak convergence in primal and dual
spaces gives us some freedom to define not only one, but several 
types of AKKT regularity.
Specifically, we will define three kinds of 
AKKT regularity
which turn out to be
satisfied in different situations and which have significantly different applications.

The organization of the paper is as follows: \Cref{Sec:Prelims} recalls some
basic definitions and preliminary results. \Cref{Sec:AKKT} introduces the notion of asymptotic
KKT conditions and essentially shows that, in a reflexive Banach space, every local minimizer of \eqref{Eq:Opt}
satisfies these AKKT conditions under mild additional assumptions on the problem's initial data. 
This result does not require any constraint qualification.
We then show in \Cref{Sec:CCP} that these AKKT conditions reduce to the usual KKT conditions
if and only if suitable sequential constraint qualifications hold which we will call 
\emph{AKKT regularity}
in our context. 
The relation between the introduced types of AKKT regularity as our new constraint qualifications and 
existing CQs in Banach spaces is discussed in \Cref{Sec:Relation}. 
\Cref{sec:CCP_in_practice} is devoted to the investigation of particular constraint systems where the introduced 
types of AKKT regularity
are inherent or can be checked with the aid of reasonable conditions. 
Particularly, we consider linear and nonlinear equality constraints as well as two-sided box constraints in Lebesgue
spaces.
An application of our results to the convergence
of augmented Lagrangian methods is presented in \Cref{Sec:Apps}. We close the paper with some
final remarks in \Cref{Sec:Final}.

\section{Preliminaries}\label{Sec:Prelims}

We mainly use standard notation in this manuscript, see \cite{Bonnans2000}.
The tools from variational analysis which we exploit here are taken from 
\cite{Aubin2009,Clarke1983,Mordukhovich2006}.

Throughout the paper, we denote strong, weak, and weak* convergence of sequences 
by $\to$, $\wto$, and $ \wtos$, respectively.
Let $X$ be a Banach space equipped with norm $\norm{\cdot}_X$.
The associated dual pairing will be represented by $\dual{\cdot}{\cdot}_X\colon X^*\times X\to\R$.
For $x\in X$ and $r>0$, $B_r(x)\subset X$ is used for the closed ball with center $x$ and radius $r$.
If $S\subset X$ is a nonempty subset of $X$, we denote by 
$d_S=\dist(\cdot,S)\colon X\to\R$ the distance function associated with $S$ w.r.t.\ the underlying norm.
In case where $X$ is a Hilbert space, we use $(\cdot,\cdot )_X\colon X\times X\to\R$ in order to
represent the associated inner product.
If $S$ is a nonempty, closed, convex 
subset of the Hilbert space $X$, we write $P_S\colon X\to X$ for the projection map onto $S$. 

Given Banach spaces $X$ and $Y$, a mapping $T\colon X\to Y$ is called
\emph{weak-to-weak* sequentially continuous} if it maps weakly convergent sequences to weak* convergent sequences, 
and \emph{completely continuous} if it maps weakly convergent sequences to strongly convergent sequences. 
It is well known that, given a Fréchet differentiable, completely continuous operator $T$, the Fr\'echet derivative 
$T'(x)\in\mathbb L(X,Y)$ is a completely continuous (or compact) linear operator for all $x\in X$,
see \cite[Thm.\ 1.5.1]{Emelyanov2007}. It is also possible (but slightly more involved) to give sufficient 
conditions for the complete continuity of the derivative mapping $T'\colon X\to \mathbb L(X,Y)$, see \cite{Palmer1969}.
Above, $\mathbb L(X,Y)$ denotes the Banach space of all bounded, linear operators mapping from $X$ to $Y$.
For brevity, the norm in $\mathbb L(X,Y)$ will be denoted by $\norm{\cdot}$ since the underlying spaces $X$ and $Y$
will be clear from the context.

For a Banach space $X$ and sets $A\subset X$ and $B\subset X^*$, we define
\begin{equation*}
    A^\circ :=\left\{v\in X^* \,\middle|\, \forall x\in A\colon\,\dual{v}{x}_{X}\le 0\right\}
    ,
    \qquad
    B^\circ :=\left\{x\in X\,\middle|\, \forall v\in B\colon\,\dual{v}{x}_X\le 0\right\}
\end{equation*}
which will be referred to as the \emph{polar cone} of $A$ and $B$, respectively.
Clearly, $A^\circ$ and $B^\circ$ are both closed, convex cones.

For a closed, convex set $S\subset X$, the  \emph{recession cone} of $ S $ is denoted by
\begin{equation*}
	S_\infty := 
	\{
	d \in X \,|\, \{d\} +S \subset S
	\}.
\end{equation*}
It can be seen easily that $ S_ \infty $ is a closed, convex cone.
Next, fix a reference point $\bar x\in S$. We denote by
\begin{equation*}
   \Rad_{S}(\bar x):=\{\alpha(x-\bar x)\,|\, \alpha\ge 0,~x\in S\},
   \qquad
   \NN_S(\bar x):=\Rad_S(\bar x)^\circ
\end{equation*}
the \emph{radial cone} (also called the \emph{cone of feasible directions}) and
the \emph{normal cone} (in the sense of convex analysis) to $S$ at $\bar x$, respectively.
For points $\tilde x\notin S$, we set $\NN_S(\tilde x):=\varnothing$.

The subsequently stated lemma relates the normal cone and the recession cone of a closed, convex set.
\begin{lemma} \label{lem:inPolarRecessionCone}
Let $ X $ be a Banach space and let $ S \subset X $ be a nonempty, closed, convex set. 
Then $ \{ v \in X^* \mid \sup_{ x\in S } \dual{v}{x}_X < \infty  \} \subset (S_\infty)^\circ $ holds. 
Particularly, $ \Nor_{S}(\bar x) \subset (S_\infty)^\circ $ is valid for all $\bar x\in S$.
\end{lemma}
\begin{proof}
Let $ v \in X^* $ be a point with $ \dual{v}{x}_X \leq c $ for some $ c \in \R $ and all 
$ x \in S $. Fix $ d \in S_\infty $ and choose an arbitrary element $ \bar x \in S $. Then $ \bar x + t d \in S $ 
is valid for all $ t > 0 $ since $ S_\infty $ is a cone. Thus $ \dual{v}{\bar x+td}_X \leq c $ 
has to hold for all $t>0$. Clearly, this implies $\dual{v}{d}_X\leq 0$. 
Since $d\in S_\infty$ was arbitrarily chosen, $v \in (S_\infty)^\circ $ follows.
\end{proof}

Given a possibly
nonconvex, closed set $ Q \subset X $ and an element $ \bar x \in Q $, we call
\begin{align*}
	\Tng_Q(\bar x)
	&:=
	\left\{
		d\in X\,\middle|\,
		\exists\{x^k\}\subset Q\,\exists\{t^k\}\subset\R\colon\,
			x^k\to\bar x,\,t^k\searrow 0,\,(x^k-\bar x)/t^k\to d
	\right\},
	\\
	\Tng^w_Q(\bar x)
	&:=
	\left\{
		d\in X\,\middle|\,
		\exists\{x^k\}\subset Q\,\exists\{t^k\}\subset\R\colon\,
			x^k\to\bar x,\,t^k\searrow 0,\,(x^k-\bar x)/t^k\wto d
	\right\},
	\\
	\Tng^\textup{C}_Q(\bar x)
	&:=
	\left\{
		d\in X\,\middle|\,
		\begin{aligned}
			&\forall\{x^k\}\subset Q\,\forall\{t^k\}\subset\R\;\text{such that}\;x^k\to \bar x,\,t^k\searrow 0\\
			&\qquad\exists\{d^k\}\subset X\colon\,d^k\to d,\,x^k+t^kd^k\in Q\,\forall k\in\N
		\end{aligned}
	\right\}
\end{align*}
the \emph{(Bouligand) tangent cone} or \emph{contingent cone}, the \emph{weak tangent cone}, and the
\emph{Clarke tangent cone}  to $Q$ at $\bar x$, respectively. 
By definition of these cones, the inclusions 
$\Tng^\textup C_Q(\bar x)\subset\Tng_Q(\bar x)\subset\Tng^w_Q(\bar x)$ are
always satisfied. In contrast to the tangent and the Clarke tangent cone, which are always closed, the
weak tangent cone does not necessarily possess this property. Furthermore, we note that the Clarke
tangent cone is always convex. If $Q$ is convex, then all these tangent cones coincide with $\cl\Rad_Q(\bar x)$.

Next, we assume that $X$ is a reflexive Banach space.
For the nonempty, closed set $Q\subset X$ and a reference point $\bar x\in Q$, we define by
\begin{align*}
	\widehat{\mathcal N}_Q(\bar x)
	&:=
	\left\{
		v\in X^*
		\,\middle|\,
		\forall x\in Q\colon\,
		\dual{v}{x-\bar x}_X\leq\oo(\norm{x-\bar x}_X)
	\right\},\\
	\Nor^\textup{L}_Q(\bar x)
	&:=
	\left\{
		v\in X^*
		\,\middle|\,
			\exists\{x^k\}\subset Q\,\exists\{v^k\}\subset X^*\colon
			x^k\to\bar x,\,v^k\wto v,\,v^k\in\widehat{\mathcal N}_Q(x^k)\,\forall k\in\N
	\right\}
\end{align*}
the \emph{Fréchet normal cone} (or \emph{regular normal cone}) and the \emph{limiting normal cone}
(or \emph{Mordukhovich normal cone}) to $Q$ at $\bar x$, respectively. 
We always have $\widehat{\Nor}_Q(\bar x)=\Tng^w_Q(\bar x)^\circ$ which is why the Fréchet normal
cone is always closed and convex. On the contrary, the limiting normal cone does not possess any
of these properties in general. Note that the above re\-presentation of the limiting normal cone
only holds in the setting of reflexive Banach spaces. 
A more general definition can be found in \cite{Mordukhovich2006}. 
If $Q$ is convex, it holds
$\Nor_Q(\bar x)=\widehat{\Nor}_Q(\bar x)=\Nor^\textup L_Q(\bar x) = \{ v \in X^* \mid 
\dual{v}{x - \bar x} \leq 0 \text{ for all }  x \in Q \} $, i.e., all these cones coincide
with the classical normal cone from convex analysis.

Let $ \bar x \in \Feas $ be a feasible point of the optimization problem \eqref{Eq:Opt}. Then
\begin{equation}\label{eq:linearization_cone}
   \Lin_\Feas (\bar x) 
   := 
   \left\{ 
   	d \in \Tng_C (\bar x) \, \middle| \, G'(\bar x) d \in \Tng_K (G(\bar x)) 
   \right\}
\end{equation}
is called the \emph{linearization cone} to $ \Feas $ at $ \bar x $. 
Note that the definition of this cone heavily depends on the precise (nonlinear)
description of the set $\Feas$ via $C$, $K$, and $G$.
One can easily check that the inclusion $\Tng_\Feas^w(\bar x)\subset\Lin_\Feas(\bar x)$
is generally valid, see, e.g., \cite[proof of Lem.\ 4.2]{GouldTolle1975}. In order to
guarantee validity of the converse inclusion, a constraint qualification is necessary
in general, see \Cref{Sec:Relation}.

We now turn to the optimality conditions of the optimization problem \eqref{Eq:Opt}. To this end, 
we define the \emph{Lagrange function} or \emph{Lagrangian} $ \Lag\colon X\times \SpaceY^* \to\R$ of the problem as
\begin{equation*}
	\forall x\in X\,~\forall\lambda\in\SpaceY^*\colon\quad
    \Lag (x,\lambda):=f(x)+\dual{\lambda}{G(x)}_{\SpaceY}.
\end{equation*}
This function occurs quite prominently in the KKT conditions of \eqref{Eq:Opt}. 
\begin{definition}\label{Dfn:KKT}
A feasible point $\bar{x}\in \Feas$ of \eqref{Eq:Opt} is called a \emph{KKT point}
if there exists $\bar\lambda\in\SpaceY^*$ such that
\begin{equation*}
	-\Lag'_x(\bar{x},\bar{\lambda})\in \Nor_{\ConX}(\bar{x})
	\quad\text{and}\quad
	\bar{\lambda}\in\Nor_{\ConY}(G(\bar{x})).
\end{equation*}
In this case, $\bar\lambda$ is called a (Lagrange) \emph{multiplier} of \eqref{Eq:Opt}
associated with $\bar x$.
\end{definition}

\section{The Asymptotic KKT Conditions}\label{Sec:AKKT}

The following is the central definition of this section. It generalizes the known definitions of
asymptotic or approximate KKT conditions from the finite-dimensional setting, see 
\cite{Andreani2010, Andreani2011,AndreaniMartinezRamosSilva2016, Andreani2019, Birgin2014}, to our 
optimization problem in Banach spaces \eqref{Eq:Opt}. 
Note that there exist different possibilities for such a 
generalization, but we found the following one particularly useful (this definition is 
essentially taken from the PhD thesis \cite{Steck2018}).

\begin{definition}\label{Def:AKKTseq}
A sequence $ \{(x^k, \lambda^k)\} \subset \ConX \times Y^* $ is called a 
\emph{strong asymptotic KKT sequence (s-AKKT sequence)} 
if there are sequences $ \{\varepsilon^k\} \subset X^*$ and $ \{r^k\} \subset [0,+\infty) $ 
such that
\begin{equation} \label{Eq:akktSeqDef}
	\forall k\in\N\colon\quad
	\varepsilon^k -\Lag'_x(x^k,\lambda^k)\in \Nor_{\ConX}(x^k)
	\quad \text{and} \quad
	\dual{ \lambda^k }{ y - G(x^k) }_{\SpaceY} \leq r^k \quad \forall y \in \ConY,
\end{equation}
with $ \varepsilon^k \to 0 $ and $ r^k \searrow 0 $. We call $ \{(x^k, \lambda^k)\} \subset \ConX \times Y^* $ 
a \emph{weak* asymptotic KKT sequence (w-AKKT sequence)} if \eqref{Eq:akktSeqDef} holds with 
$ \varepsilon^k \rightharpoonup^* 0 $ and $ r^k \searrow 0 $.
\end{definition}

Note that the first condition simplifies to $ \varepsilon^k -\Lag'_x(x^k,\lambda^k) = 0 $ if $ \ConX = X $,
whereas the second condition implies $ \lambda^k \in \ConY^\circ $ for all $ k \in \mathbb{N} $ if $ \ConY $ is a 
cone. 
To see the latter statement, fix $ k \in \mathbb{N} $. Then, exploiting the fact that $ \ConY $ is a cone,
the condition $ \dual{ \lambda^k }{ y - G(x^k) }_{\SpaceY} \leq r^k  $ for all $ y \in \ConY $ can be written as
$ \dual{ \lambda^k }{ \alpha y - G(x^k) }_{\SpaceY} \leq r^k  $ for all $ y \in \ConY $ and all $ \alpha > 0 $. Dividing
this expression by $ \alpha > 0 $ yields $ \dual{ \lambda^k }{  y - \frac{G(x^k)}{\alpha} }_{\SpaceY} \leq\frac{ r^k}{\alpha} $
for all $ y \in \ConY $ and all $ \alpha > 0 $. Taking the limit $ \alpha \to \infty $ therefore implies
$ \dual{ \lambda^k }{  y }_{\SpaceY}\leq 0 $ for all $ y \in \ConY $. Hence, $ \lambda^k \in \ConY^\circ $ 
holds, see \Cref{rem:K_cone} as well.

The main idea of (strong or weak) AKKT sequences is, obviously, the existence of a sequence which
satisfies the KKT conditions inexactly using a certain measure of inexactness. 
As already observed,
e.g., in \cite{Andreani2019}, there exist different ways to measure the degree of inexactness, and these
measures are not necessarily equivalent, even in finite dimensions. As an example, consider the usual
complementarity condition $ \lambda \geq 0$, $g(x) \leq 0$, $\lambda g(x) = 0 $ associated with
an inequality constraint $ g(x) \leq 0 $ which is induced by a function $g\colon\R^n\to\R$. 
This complementarity condition can be rewritten as
$ \min ( - g(x), \lambda ) = 0 $, hence, the condition $ | \min ( - g(x), \lambda ) |
\leq r $ is a very natural criterion for an inexact satisfaction of the complementarity
condition. Alternatively, one might use a condition like $ \lambda \geq 0$, $g(x) \leq 0$, 
$- \lambda g(x) \leq r $. These two conditions, however, are not equivalent, e.g.,
for $n=1$, take $ g(x) := - x $ and consider the sequences defined by $ r^k := 1/k, \lambda^k := k^2 $, 
and $ x^k := 1/k $ for each $k\in\N$. Then the first condition holds for all $ k\in\N $, whereas the second one is 
violated, in fact, $ - \lambda^k g(x^k) \to \infty $. This should be kept in mind because
the concise definition of a (weak or strong) AKKT sequence plays a crucial role. In this paper,
we take advantage of \Cref{Def:AKKTseq}, but alternative definitions might also be useful in
other contexts. The above example also depicts that our definition of a (weak or strong) AKKT sequence is slightly
different from the one stated in \cite[Definition~1.2]{AndreaniMartinezRamosSilva2016} for standard
nonlinear problems in finite dimensions.

Note that \Cref{Def:AKKTseq} does not require any convergence or weak convergence of the sequence
$ \{ ( x^k, \lambda^k ) \} $. By forcing weak or strong convergence of the primal sequence $ \{ x^k \} $, we
obtain the following definition. Notice that this definition still does not assume any convergence or
boundedness of the dual sequence of multipliers $ \{ \lambda^k \} $.

\begin{definition}\label{Def:AKKTpoint}
Let $ \bar x \in \Feas $ be a feasible point of \eqref{Eq:Opt}. Then we call $ \bar x $ a
\begin{enumerate}[(a)]
   \item \emph{weak asymptotic KKT point (w-AKKT point)} if there exists a w-AKKT sequence 
      $ \{ (x^k, \lambda^k) \} $ such that $ x^k \wto \bar x$.
   \item \emph{strong asymptotic KKT point (s-AKKT point)} if there exists an s-AKKT sequence 
      $ \{ (x^k, \lambda^k) \} $ such that $ x^k \to \bar x$.
\end{enumerate}
\end{definition}

In principle, there exist four possible definitions of asymptotic KKT points due to the four possible
combinations of strong and weak convergence of the underlying sequences $ \{ x^k \} $ and 
$ \{ \varepsilon^k \} $. Taking this into account, a more precise terminology for 
w- and s-AKKT points would be ww-AKKT and ss-AKKT points. In addition, one could
therefore also define sw-AKKT and ws-AKKT points, and we do not exclude the possibility
that these additional notions might be useful in certain situations. For our purposes, however,
the above two definitions are sufficient, and to avoid an overkill in the terminology, we simply
talk about w-AKKT and s-AKKT points. In particular, the notion of w-AKKT points is motivated by the 
fact that suitable methods generate sequences which admit weak accumulation points, not necessarily 
strong ones. 
 
The following statement, inspired by the corresponding finite-dimensional result in 
\cite[Thm.\ 3.1]{Birgin2014} and the Hilbert space result in \cite[Prop.\ 3.50]{Steck2018}, essentially
shows that every local minimizer of \eqref{Eq:Opt} is an s-AKKT point, therefore, in particular, a
w-AKKT point. Recall that a local minimizer is not necessarily a KKT point, hence, the
concept of (strong or weak) AKKT points is more general than the notion of KKT points.

\begin{proposition}\label{Prop:OptAKKT2}
Let $\bar{x}$ be a local minimizer of \eqref{Eq:Opt}.
Assume that $X$ is reflexive, and suppose that $f$ is weakly sequentially lower semicontinuous in a neighborhood of $\bar x$.
Moreover, we assume that
\begin{equation}\label{eq:continuity_property_distance_constraint_map}
	\forall\{x^k\}\subset C\,\forall x\in C\colon\qquad
	x^k \wto x
	\quad\text{and}\quad
	d_K(G(x^k)) \to 0
	\quad\Longrightarrow\quad
	G(x) \in K
\end{equation}
holds.
Then $ \bar x $ is an s-AKKT point and, thus, 
also a w-AKKT point of \eqref{Eq:Opt}.
\end{proposition}
\begin{proof}
	Let $\epsilon>0$ be such that $\bar{x}$ minimizes $f$ on $B_\epsilon(\bar{x})\cap\Feas$. 
	Noting that $f$ is continuous, we find some $r\in(0,\epsilon)$ such that $f$ 
	is bounded from below
	and weakly sequentially lower semicontinuous
	on $B_r(\bar x)\cap\ConX$.
	For $k\in\N$, consider the problem
	\begin{equation}\label{Eq:AKKT:PropOptAKKT1}
		\begin{aligned}
			&\minimize_{x\in X}&& f(x)+\norm{x-\bar{x}}_X^2+k d_{\ConY}^2(G(x))&
			&\subjectto&& x\in B_r(\bar{x})\cap \ConX.&
		\end{aligned}
	\end{equation}
	An application of Ekeland's variational principle
	\cite[Thm.\ 3.3.1]{Aubin2009}
	in the complete metric space $X$
	yields a point $x^k \in B_r(\bar x) \cap \ConX$
	which minimizes
	\begin{equation}
		\label{Eq:OptTheory:VI:AKKT:PropOptAKKT3}
		\begin{aligned}
			&\minimize_{x\in X}&& f(x)+\norm{x-\bar{x}}_X^2+k d_{\ConY}^2(G(x))
				+ \frac1k \, \norm{x - x^k}_X&\\
			&\subjectto&& x\in B_r(\bar{x})\cap \ConX.&
		\end{aligned}
	\end{equation}
	By using $x = \bar x$, we obtain
	\begin{equation}
		\label{eq:from_ekeland}
		f(x^k)+\norm{x^k-\bar{x}}_X^2+k d_{\ConY}^2(G(x^k))
		\le
		f(\bar x)
		+ \frac1k \, \norm{\bar x - x^k}_X
		.
	\end{equation}
	Since $\{x^k\}$ is bounded,
	we have w.l.o.g.\ $x^k \wto \hat x$
	for some $\hat x \in B_r(\bar x) \cap C$.
	From \eqref{eq:from_ekeland} we obtain
	\begin{align*}
		f(\bar x)
		&\ge
		\limsup_{k \to \infty} \paren[\big]{ f(x^k) + \norm{x^k - \bar x}_X^2 + k \, d_K^2(G(x^k))}
		\\
		&\ge
		\liminf_{k \to \infty} f(x^k) + \limsup_{k \to \infty} \norm{x^k - \bar x}_X^2 + 
		\liminf_{k \to \infty}k \, d_K^2(G(x^k))
		\\
		&\ge
		f(\hat x) + \limsup_{k \to \infty} \norm{x^k - \bar x}_X^2 + 
		\liminf_{k \to \infty}k \, d_K^2(G(x^k)) .
	\end{align*}
	This shows that $d_K(G(x^k)) \to 0$ holds (along a subsequence).
	Hence, \eqref{eq:continuity_property_distance_constraint_map} implies 
	validity of $G(\hat x) \in K$.
	Thus, $\hat x \in \Feas$.
	Now, the above inequality implies
	\begin{equation*}
		f(\bar x)
		\ge
		f(\hat x) + \limsup_{k \to \infty} \norm{x^k - \bar x}_X^2
		\ge
		f(\bar x) + \limsup_{k \to \infty} \norm{x^k - \bar x}_X^2.
	\end{equation*}
	Thus, $x^k \to \bar x$ and, therefore, $\hat x = \bar x$.

	Recalling that $x^k$ is a solution of \eqref{Eq:OptTheory:VI:AKKT:PropOptAKKT3} for each $k\in\N$
	while $x^k\to\bar x$ holds, 
	$\norm{x^k-\bar x}_X<r$ is valid for sufficiently large $k\in\N$.
	Thus, we obtain the existence of sequences $\{\varepsilon^k\}\subset X^*$ 
	and $\{\xi^k\}\subset Y^*$
	such that $-\varepsilon^k\in\partial \norm{\cdot}_{X}^2(x^k-\bar{x})+\tfrac1k\partial\norm{\cdot}_X(0)$ 
	and $\xi^k\in \partial d_K(G(x^k))$ as well as
	\begin{equation}\label{eq:optimality_condition_penalized_problem}
		0\in \left\{f'(x^k)-\varepsilon^k +2 k\,d_K(G(x^k))\, G'(x^k)^*\xi^k\right\}+ \Nor_{\ConX}(x^k)
	\end{equation}
	hold. Here, we exploited the Fermat, sum, as well as chain rule for Clarke's generalized derivative,
	see \cite[Prop.\ 2.3.2, 2.3.3, Thm.\ 2.3.9, 2.3.10]{Clarke1983}.
	Similarly, Clarke's chain rule implies
	$\partial\norm{\cdot}_X^2(x^k - \bar x) = 2 \, \norm{x^k - \bar x}_X\,\partial\norm{\cdot}_X(x^k - \bar x)$.
	Moreover, it is well known that the inclusion $\partial\norm{\cdot}_X(x) \subset B_1(0)$ holds 
	in $X^*$ for all $x \in X$.
	Hence, we obtain $\varepsilon^k\to 0$ in $X^*$. 
	By definition of the convex function $d_K$ and its subdifferential at $G(x^k)$, we have
	\[
		0=d_K(y)\geq d_K(G(x^k))+\dual{ \xi^k }{ y-G(x^k)}_{\SpaceY}\geq \dual{\xi^k}{y-G(x^k)}_{\SpaceY}
	\] 
	for all $y\in\ConY$. 
	Thus, setting $\lambda^k:=2 k\, d_K(G(x^k))\, \xi^k$, we have
	$\dual{ \lambda^k}{y-G(x^k)}_{\SpaceY}\le 0$ for all $y\in\ConY$.
	Combining this with \eqref{eq:optimality_condition_penalized_problem}, we easily see
	that $\bar x$ is an s-AKKT point.
\end{proof}

Let us briefly note that the continuity property \eqref{eq:continuity_property_distance_constraint_map}
obviously holds whenever the composition $d_K\circ G$ is weakly sequentially lower semicontinuous.
In particular, this property is inherent whenever $ G $ is an affine mapping induced by a bounded
linear operator since then the composition $d_K\circ G$ is a continuous, convex mapping.

\begin{remark}
	If $ X $ is a Hilbert space (or if $\partial \norm{ \cdot }_X^2$ is strongly monotone)
	and $ f^\prime $ is locally Lipschitz continuous,
	then the weak sequential lower semicontinuity of $ f $  can be omitted, by using the problem
	\begin{equation*}
		\begin{aligned}
			&\minimize_{x\in X}&& f(x)+ \gamma \norm{x-\bar{x}}_X^2+k d_{\ConY}^2(G(x))
			&
			&\subjectto&& x\in B_r(\bar{x})\cap \ConX&
		\end{aligned}
	\end{equation*}
	instead of \eqref{Eq:AKKT:PropOptAKKT1}. Above,
	$\gamma  $ is sufficiently large such that 
	$X\ni x\mapsto f(x)  + \gamma \norm{ x - \bar x }_X^2\in\R $ 
	is locally convex and, thus, weakly sequentially lower semicontinuous 
	(this statement can be verified, e.g., by exploiting the equivalence of monotonicity of the gradient 
	and convexity of the underlying function).
\end{remark} 

We close this section with two illustrative examples. The first one considers an optimization problem
which possesses a minimizer where the KKT conditions are violated, but this minimizer is an s-AKKT point.

\begin{example}\label{Ex:MinNoKKT}
Consider the optimization problem \eqref{Eq:Opt} with
\begin{align*}
	X &:= \R \times L^2(0,1),
	\qquad
	Y := L^2(0,1),
	&
	K &:= \{0\} \subset Y,
	\\
	C &:= \R \times \set{u \in L^2(0,1) \given -1 \le u \le 1},
	&
	G(\alpha,u) &:= \alpha \cdot q - u,
\end{align*}
where $q \in L^2(0,1) \setminus L^\infty(0,1)$ is a fixed function.

We argue that $\Feas = \{0\}$ holds. Indeed, the constraint $G(\alpha,u) \in K$
implies $u = \alpha \, q$. If $\alpha \ne 0$, $u$ is unbounded and, therefore, $(\alpha,u) \not\in C$.
Thus, $(\alpha^*, u^*) = (0,0)$ is the global minimizer of \eqref{Eq:Opt} for any objective function $f$.
In particular, $f'(\alpha^*, u^*)$ can be an arbitrary vector in $X\dualspace = \R \times L^2(0,1)$.

It is easy to check the validity of
$\TT_C(0,0) = \R \times L^2(0,1)$,
$ \TT_K(0) = \RR_K(0) = \{0\} $,
$ \NN_C(0,0) = \{(0,0)\} $,
and
$ \NN_K(0) = Y\dualspace$.
Hence, $(0,0)$ is a KKT point if and only if
\begin{equation*}
	f'(0,0)
	=
	-G'(0,0)\adjoint \lambda
	=
	\begin{pmatrix}
		-\dual{q}{\lambda}_{L^2(0,1)}
		\\
		\lambda
	\end{pmatrix}
\end{equation*}
holds for some $\lambda \in L^2(0,1)$. Thus, there exist linear functionals $f$
such that $(0,0)$ is not a KKT point,
e.g., $f := (-1,0)\in X^*$.

Let us show that $(0,0)$ is an s-AKKT point for the linear functional $f = (-1,0)$.
Of course, this follows from \Cref{Prop:OptAKKT2}, but it is instructive
to construct the corresponding s-AKKT sequence explicitly.
To this end, we fix the function $q$ via $q(t) = t^{-1/4}$. Now, we set $\alpha^k := 1/k$
and $u^k(t) := P_{[-1,1]}(\alpha^k \, q(t))$, i.e.,
\begin{equation*}
	u^k(t)
	=
	\begin{cases}
		1 & \text{for } t \le k^{-4}, \\
		\frac{t^{-1/4}}{k} & \text{for } t > k^{-4}.
	\end{cases}
\end{equation*}
Next, we choose $\lambda^k \in L^2(0,1)$ which is supported on $[0,k^{-4}]$
with $\dual{q}{\lambda^k}_{L^2(0,1)} = 1$, e.g., $\lambda^k = \frac34 \, k^{3} \, \chi_{[0,k^{-4}]}$.
Thus, $\mu^k := (0,\lambda^k)$ is an element of $\NN_C(\alpha^k, u^k)$. Moreover, we have
\begin{equation*}
	f'(\alpha^k, u^k)
	+
	G'(\alpha^k,u^k)\adjoint \lambda^k + \mu^k
	=
	\begin{pmatrix}
		-1 \\ 0
	\end{pmatrix}
	+
	\begin{pmatrix}
		\dual{q}{\lambda^k}_{L^2(0,1)} \\ -\lambda^k
	\end{pmatrix}
	+
	\begin{pmatrix}
		0 \\ \lambda^k
	\end{pmatrix}
	=
	0
\end{equation*}
and
\begin{equation*}
	\dual{\lambda^k}{0 - G(\alpha^k, u^k)}_{L^2(0,1)} =-(4k)^{-1}\le 0.
\end{equation*}
Thus, $\{((\alpha^k, u^k), \lambda^k)\}$ is an s-AKKT sequence.
Due to $\norm{\lambda^k}_{L^2(0,1)}=\tfrac34k$, $ \{ \lambda^k \} $ is unbounded in $L^2(0,1)$.

Note that, for this choice of $f$, the point $(0,0)$
is not a KKT point and not even a Fritz--John point
of the associated problem \eqref{Eq:Opt}.
\hfill$\Diamond$
\end{example}

It might be possible to construct a similar example involving $\ell^2$
by using an idea of \cite[Ex.\ 2.5]{Penot1981}.
The second example indicates that \Cref{Prop:OptAKKT2} may not hold without 
the reflexivity of the underlying space.

\begin{example}\label{Ex:necessity_of_reflexivity}
Consider again the optimization problem \eqref{Eq:Opt} with the data
\begin{align*}
	X &:= \ell^1, 
	&
	Y &:= \ell^2,
	&
	C &:= \ell^1,
	\\
	K &:= \{0\} \subset \ell^2,
	&
	G(x) &:= x,
	&
	f(x) &:= \sum_{i = 1}^\infty a_i \, x_i
\end{align*}
for some given sequence $a \in \ell^\infty \setminus c_0$.
Clearly, $\bar x := 0$ is the only feasible point, and therefore also optimal.

We argue by contradiction. Let us assume that $\{(x^k, \lambda^k)\}$ is an s-AKKT sequence. 
Then the convergence
\begin{equation*}
	f'(x^k) + G'(x^k)\adjoint \lambda^k
	=
	a + \lambda^k
	\to 0
\end{equation*}
has to hold in $X^*=\ell^\infty$.
However, it holds $\lambda^k \in \ell^2 \subset c_0$ for each $k\in\N$, and
$c_0$ is a closed subspace of $\ell^\infty$. On the other hand,
 $a \not \in c_0$ holds by assumption. This contradiction 
shows that the reflexivity assumption in \Cref{Prop:OptAKKT2} is essential. 
\hfill$\Diamond$
\end{example}

\section{Asymptotic KKT Regularity}\label{Sec:CCP}

The previous section was devoted to the notion of (strong or weak) AKKT points. In particular,
under fairly mild assumptions, we proved that any local minimizer of \eqref{Eq:Opt} 
is a (strong) AKKT point. Hence,
$ \bar x $ being an AKKT point is a necessary optimality condition which, in particular,
does not require any constraint qualification. Therefore, the natural question arises under
which assumption such an AKKT point is already a KKT point. By generalization of the corresponding 
finite-dimensional theory from \cite{AndreaniMartinezRamosSilva2016,Andreani2019} to our setting, 
this leads to three types of 
asymptotic KKT regularity 
which turn out to be constraint qualifications. 
More precisely, we will see that, in some sense, they are the weakest possible constraint
qualifications which guarantee that a given (strong or weak) AKKT point 
of \eqref{Eq:Opt} is already a KKT point. 
The relation of these 
types of AKKT regularity
to some existing constraint qualifications will be discussed in
\Cref{Sec:Relation}.

Motivated by the definition of (strong or weak) AKKT points, let us introduce
\begin{equation}\label{Eq:SeqMxr}
   \M (x,r) 
   := 
   \left\{
   G^\prime (x)^* \lambda + \mu \in X^*
   \, \middle|\, 
   \begin{aligned}
   	&\lambda\in\SpaceY^*,\,\mu\in\NN_C(x),\\
   	&\dual{ \lambda }{ y - G(x) }_{\SpaceY} \leq r \, \forall y \in \ConY
   \end{aligned}
   \right\}
\end{equation}
for $x \in X$ and $r \geq 0 $. 
Note that we have $\M(\tilde x,r)=\varnothing$ for all $\tilde x\notin C$ and
$r \geq 0  $ by definition of the normal cone.
If $\bar x\in X$ is feasible to \eqref{Eq:Opt}, it holds that
\begin{equation}
	\label{Eq:SeqMbarx0}
	\M(\bar x,0)
	=
	\left\{
		G'(\bar x)^*\lambda+\mu\in X^*
		\,\middle|\,
		\lambda\in\NN_K(G(\bar x)),\,\mu\in\NN_C(\bar x)
	\right\}.
\end{equation}
It follows that the condition $ - f'(\bar x) \in \M(\bar x,0) $ 
is equivalent to $ \bar x$ being a KKT point of \eqref{Eq:Opt},
see \cref{Dfn:KKT}.

Let us first state two simple observations 
regarding the set $ \M (x,r)  $.
\begin{remark}\label{rem:K_cone}
	If $\ConY$ is a cone, we have the equivalence
	\[
		\forall y \in \ConY\colon\,\dual{ \lambda }{ y - G(x) }_{\SpaceY} \leq r  
		\quad\Longleftrightarrow\quad
		\lambda\in \ConY^\circ,\;-\dual{\lambda}{G(x)}_{\SpaceY}\leq r
	\]
	for each $\lambda\in Y^*$, cf.\ the corresponding discussion after
	\Cref{Def:AKKTseq}. This yields the representation
	\[
		\M (x,r)
		=
		\left\{
			G^\prime(x)^*\lambda + \mu\in X^*
			\mid
			\lambda\in \ConY^\circ ,\, \mu \in \NN_C(x),\,
			-\dual{\lambda}{G(x)}_{\SpaceY}\leq r
		\right\}
	\]
	for all $x\in X$ and $r \geq 0 $.
\end{remark}

\begin{remark}
	For a general convex set $ K $, the condition
	\begin{equation*}
		\forall y \in K\colon\quad
		\dual{ \lambda }{ y - G(x) }_{\SpaceY} \leq r  
	\end{equation*}
	from the definition of $ \M (x,r) $ implies that 
	$ \sup_{ y \in K } \dual{ \lambda }{ y }_{\SpaceY} < +\infty $,
	hence $ \lambda \in (K_\infty) ^\circ $, by \Cref{lem:inPolarRecessionCone}.
	This can be interpreted as a sign property of the Lagrange multiplier $ \lambda $, as it was 
	introduced in \cite{Andreani2019} in the finite-dimensional setting.
\end{remark}

Before we introduce the three already advertised versions of AKKT regularity, let us define the following
Painlevé--Kuratowski-type
outer/upper limits
\begin{align*}
	\limsup_{
		\substack{x \to \bar x \\ r \searrow 0}
	} \M(x,r)
	&:=
	\left\{
	\bar v \in X^* \,\middle| \,
		\begin{aligned}
			&\exists \{x^k\}\subset X\,\exists\{r^k\}\subset [0,+\infty) \,
			\exists\{v^k\}\subset X^*\colon\\
			&\qquad  x^k \to \bar x,\, r^k \searrow 0 ,\, v^k \to \bar v,\\
			&\qquad  v^k \in \M(x^k , r^k )\,\forall k\in\N
		\end{aligned}
	\right\},
	\\
	\wStarlimsup_{
			\substack{x \to \bar x \\ r \searrow 0}
	} \M(x,r)
	&:=
	\left\{
	\bar v \in X^* \,\middle| \,
		\begin{aligned}
			&\exists \{x^k\}\subset X\,\exists\{r^k\}\subset [0,+\infty) \,
			\exists\{v^k\}\subset X^*\colon\\
			&\qquad x^k \to \bar x ,\,r^k \searrow 0 ,\, v^k \wtos \bar v,\\
			&\qquad  v^k \in \M(x^k , r^k )\,\forall k\in\N
		\end{aligned}
	\right\},
	\\
	\wStarlimsup_{
			\substack{x \wto \bar x \\ r \searrow 0}	
	} \M(x,r)
	&:=
	\left\{
	\bar v \in X^* \,\middle| \,
		\begin{aligned}
			&\exists \{x^k\}\subset X\,\exists\{r^k\}\subset [0,+\infty) \,
			\exists\{v^k\}\subset X^*\colon\\
			&\qquad  x^k \wto \bar x,\, r^k \searrow 0 ,\, v^k \wtos \bar v,\\
			&\qquad v^k \in \M(x^k , r^k )\,\forall k\in\N
		\end{aligned}
	\right\}
\end{align*}
of the set-valued map $ \M \colon X \times [0,+\infty) \rightrightarrows X\dualspace $ w.r.t.\ some point 
$\bar x\in\Feas$. Recall that $ \M(x^k,r^k) $ is empty for all $ x^k \not\in C $, hence, requiring
the existence of a sequence $ \{ x^k \} \subset X $ in the previous definitions is equivalent to 
assuming that this sequence belongs to the set $ C $.

	These definitions give rise to the following observation 
	\begin{lemma}
		\label{lem:referee}
		Let $\bar x \in \Feas$ be a feasible point of \eqref{Eq:Opt}.
		% Then the following statements hold.
		\begin{enumerate}[label=(\alph*)]
			\item
				Then $\bar x$ is an s-AKKT point if and only if 
				$-f'(\bar x) \in \limsup_{ \substack{x \to \bar x \\ r \searrow 0}	} \M(x,r)$.
			\item
				Suppose that $f' \colon X \to X\dualspace$ is weak-to-weak*-sequentially continuous.
				Then $\bar x$ is a w-AKKT point if and only if 
				$-f'(\bar x) \in \wStarlimsup_{ \substack{x \wto \bar x \\ r \searrow 0}	} \M(x,r)$.
		\end{enumerate}
	\end{lemma}
	The proof follows immediately from the definitions.
	We already observed that $\bar x$ is a KKT point if and only if
	$-f'(\bar x) \in \MM(\bar x, 0)$.
	This motivates the next definition.

\begin{definition}\label{Def:AKKTCQ}
Let $ \M (x,r) $ be defined as in \eqref{Eq:SeqMxr}, and let $ \bar x \in\Feas$ be any feasible point of 
\eqref{Eq:Opt}. Then $ \bar x $ is
\begin{enumerate}[(a)]
   \item \emph{strongly AKKT regular} (s-AKKT regular) if
	\begin{equation*}
	\limsup_{
		\substack{x \to \bar x \\ r \searrow 0}	
	} \M(x,r)
	\subset
	\M( \bar x , 0 );
	\end{equation*}
   \item \emph{strongly-weakly AKKT regular} (sw-AKKT regular) if
	\begin{equation*}
	\wStarlimsup_{
		\substack{x \to \bar x \\ r \searrow 0}
	} \M(x,r)
	\subset
	\M( \bar x , 0 );
	\end{equation*}
   \item \emph{weakly AKKT regular} (w-AKKT regular) if
	\begin{equation*}
		\wStarlimsup_{
	\substack{x \wto \bar x \\ r \searrow 0}	
	} \M(x,r)
	\subset
	\M( \bar x , 0 ).
	\end{equation*}
\end{enumerate}
\end{definition}

Note that these three conditions are equivalent in finite dimensions.
In the infinite-dimensional setting, however, we only have
\begin{equation*}
	\text{w-AKKT regular}
	\quad\Longrightarrow\quad
	\text{sw-AKKT regular}
	\quad\Longrightarrow\quad
	\text{s-AKKT regular}
	.
\end{equation*}
While the role of 
s- and w-AKKT regularity 
is motivated by
the definition of s-AKKT and w-AKKT points, respectively,
we will see in \Cref{Sec:Relation}
that sw-AKKT regularity possesses reasonable relationships
to classical CQs in infinite-dimensional programming. 
Let us also emphasize
that \Cref{Def:AKKTCQ} is stated for feasible points only. Hence, whenever
we use one of the above AKKT regularity conditions in our subsequent analysis,
there is the implicit assumption that $ \bar x $ is feasible to \eqref{Eq:Opt}
even if this might not be stated explicitly.

In the above definitions of the three types of AKKT regularity, 
it was used that $ r \geq 0 $.
This, however, was just done to make the definition more concise.
If one desires to allow $ r < 0 $, all results presented in the sequel are 
applicable using $ \tilde r := \max (0 , r ) $.

\begin{remark}\label{rem:closedness_via_CCP}
	Let $\bar x\in\Feas$ be a feasible point of \eqref{Eq:Opt}.
	Then, if any of the conditions in \Cref{Def:AKKTCQ}
	is satisfied, we automatically have equality in the defining property,
	e.g., s-AKKT regularity implies
	\[
		\limsup_{
		\substack{x \to \bar x \\ r \searrow 0}	
		} \M(x,r)
		=
		\M( \bar x , 0 )
		.
	\]
	Particularly, this yields that $\M(\bar x, 0)$ is closed.
	
	Noting that $\M(\bar x,0)$ is convex,
	$\bar x$ being s-AKKT regular 
	implies the weak closedness of
	$\M(\bar x,0)$. Whenever $X$ is reflexive, this coincides with weak* closedness of $\M(\bar x,0)$.
\end{remark}

We next discuss the question under which condition a strong or weak AKKT point is already
a KKT point. We first state a ``strong'' formulation in the following result.

\begin{theorem}\label{sAKKTpIsKKT}
Let $\bar x\in\Feas$ be a feasible point of \eqref{Eq:Opt}.
Then the following statements hold:
\begin{enumerate}[(a)]
   \item If $ \bar x $ is an s-AKKT point 
   being s-AKKT regular, 
   then $ \bar x $ is a KKT point.
   \item Conversely, if for every continuously differentiable function $ f $, the implication
		``$\bar x$ is an s-AKKT point $\Longrightarrow$ $\bar x$ is a KKT point''
		holds,
        then $ \bar x $ is s-AKKT regular.
\end{enumerate}
\end{theorem}

\begin{proof}
(a) If $ \bar x $ is an s-AKKT point of \eqref{Eq:Opt}, 
\cref{lem:referee}~(a) and the s-AKKT regularity imply
\begin{equation*}
   - f'(\bar x) 
   \in 
   \limsup_{
   \substack{x \to \bar x \\ r \searrow 0}	
   } \M(x,r)
   \subset
   \M(\bar x , 0 ).
\end{equation*}
Hence, $ \bar x $ is a KKT point. \\[-1mm]

\noindent
(b) Conversely, assume that the s-AKKT conditions at $ \bar x $ imply the KKT conditions for every
continuously differentiable objective function. Then take an arbitrary element
$ \bar v \in \limsup_{\substack{x \to \bar x \\ r \searrow 0}	} \M(x,r) $.
By definition, there exist sequences $\{x^k\}\subset X$, $\{r^k\}\subset[0, + \infty)$, and $\{v^k\}\subset X^*$
such that $ x^k \to \bar x $, $r^k \searrow 0 $, and $v^k \to \bar v $ 
as well as $ v^k \in \M(x^k , r^k ) $ for all $ k \in \N $.
Hence, there exist $ \lambda^k \in \SpaceY^* $ with $ \dual{ \lambda^k }{ y - G(x^k) }_{\SpaceY}\leq r^k $
for all $ y \in K $, and $ \mu^k \in \Nor_C (x^k) $ such that $ v^k = G'(x^k)^* \lambda^k + \mu^k $
for all $ k \in \N $. 
We emphasize that $\{x^k\}\subset C$ holds by definition of the normal cone.
Let us define the particular objective function $ f(x) := - \dual{ \bar v }{ x }_X $
for all $x\in X$. Then $ \varepsilon^k := f'(x^k) + v^k = - \bar v + v^k \to 0 $, and we have
\begin{align*}
  \varepsilon^k - L'_x(x^k, \lambda^k) & = v^k + f'(x^k) - f'(x^k) - G'(x^k)^* \lambda^k \\
  & = v^k - G'(x^k)^* \lambda^k \\
  & = \mu^k \in \Nor_C(x^k)
\end{align*}
for all $ k \in \N $. Hence, $\bar x$ is an s-AKKT point
of the associated problem \eqref{Eq:Opt}.
By assumption, it follows $\bar v = - f'(\bar x) \in \M(\bar x , 0)$. 
This clearly shows the inclusion
$ \limsup_{ \substack{x \to \bar x \\ r \searrow 0}} \M(x,r) \subset \M(\bar x , 0) $,
i.e., $ \bar x $ is s-AKKT regular.
\end{proof}

This theorem is similar to \cite[Thm.\ 3.2]{AndreaniMartinezRamosSilva2016}
which characterizes a 
AKKT-regularity-type property
 in the setting of standard
finite-dimensional nonlinear programming.
However, note that the phrase
``that attains a minimum at $x^*$''
has to be deleted from the statement of
\cite[Thm.\ 3.2]{AndreaniMartinezRamosSilva2016},
since otherwise this result would characterize the Guignard constraint qualification, 
despite the fact that the minimizer property does not hold for the function constructed in the presented proof.
Taking together the adjusted theorem from \cite{AndreaniMartinezRamosSilva2016}
and \Cref{sAKKTpIsKKT}, our 
types of AKKT regularity
from \cref{Def:AKKTCQ} are clearly closely related 
to the cone continuity property from \cite{AndreaniMartinezRamosSilva2016} in
the setting of standard nonlinear programming in finite dimensions.
However, as we already mentioned in \Cref{Sec:AKKT}, our definition of an AKKT sequence differs
slightly from the one in \cite{AndreaniMartinezRamosSilva2016} in this setting, so it remains an open problem to check whether the concepts of the cone continuity property and AKKT regularity 
actually coincide.
In contrast to \cite{AndreaniMartinezRamosSilva2016}, our set $ \M(x,r) $ is not necessarily a cone, 
which is why we abstain from referring to the constraint qualifications from \cref{Def:AKKTCQ} as
\emph{cone continuity properties}. 
Let us point out that the term \emph{AKKT regularity} was later also exploited by the authors of 
\cite{AndreaniMartinezRamosSilva2016}, too, in order to refer to constraint qualifications
of this type, see e.g.\ \cite{Andreani2019}.

The previous theorem is quite similar to the statement that Guignard's constraint qualification
is the weakest constraint qualification which ensures that local minimizers are KKT points,
see \cite{GouldTolle1975}.

Using essentially the same technique of proof, we obtain the following ``weak'' counterpart
of \Cref{sAKKTpIsKKT}. 

\begin{theorem}\label{wAKKTpIsKKT}
Let $\bar x\in\Feas$ be a feasible point of \eqref{Eq:Opt}.
Then the following statements hold:
\begin{enumerate}[(a)]
   \item Suppose that
   		 $ f' : X \to X^* $ is weak-to-weak*-sequentially continuous.
   			If $ \bar x $ is w-AKKT point 
   			being w-AKKT regular, 
   			then $ \bar x $ is a KKT point.
   \item Conversely, if for every continuously differentiable function $ f $, the implication
		``$\bar x$ is a w-AKKT point $\Longrightarrow$ $\bar x$ is a KKT point''
		holds,
		then $ \bar x $
		is w-AKKT regular.
\end{enumerate}
\end{theorem}

The previous results imply that
strong, strong-weak, and weak AKKT regularity 
are constraint qualifications
under appropriate assumptions on the initial problem data of \eqref{Eq:Opt}. 
In fact, given a local minimum $ \bar x $ of \eqref{Eq:Opt} 
which is w-, sw-, or s-AKKT regular, 
it follows from \Cref{Prop:OptAKKT2} (under the assumptions stated there)
that $ \bar x $ is an s-AKKT point. 
Noting that w- or sw-AKKT regularity 
implies
s-AKKT regularity, the first statement of \Cref{sAKKTpIsKKT} can be used
to infer that $\bar x$ is already a KKT point of \eqref{Eq:Opt}.
We summarize these observations
in the following corollary.

\begin{corollary}\label{Cor:AKKT-CQ_is_CQ}
Let $X$ be reflexive and let \eqref{eq:continuity_property_distance_constraint_map} be satisfied.
Then  w-, sw-, and s-AKKT regularity 
are constraint qualifications for \eqref{Eq:Opt} in the following sense:
For every objective $f$ which is continuously differentiable and weakly sequentially lower semicontinuous,
local optimality of $\bar x$ implies that $\bar x$ is a KKT point.
\end{corollary}

Using the terminology from \cite{AndreaniMartinezRamosSilva2016}, any condition which guarantees 
that an AKKT point is already a KKT point is called a \emph{strict constraint qualification}.
The previous results therefore show that our 
AKKT-regularity-type 
conditions are strict constraint qualifications,
and that they are the weakest possible ones.

Finally, we want to present sufficient criteria for 
sw- and w-AKKT regularity. 
To this end, recall that,
given any sequences $ \{ x^k \}\subset X$, $\{ r^k \}\subset [0,+\infty) $, and $ \{ v^k \}\subset X^* $ 
with $ v^k \in \M (x^k, r^k) $ for all $ k \in \N $, it follows that there exist corresponding sequences 
$ \{ \lambda^k \}\subset\SpaceY^* $ and $ \{ \mu^k \}\subset X^* $ such that $ v^k = G'(x^k)^* \lambda^k + \mu^k $ 
holds for all $ k \in \N $. In general, these sequences of multipliers might be unbounded even if
$\{x^k\}$, $\{r^k\}$, and $\{v^k\}$ converge.
The following result discusses the situation where the sequences $\{\lambda^k\}$ and $\{\mu^k\}$
can be chosen as bounded ones.

\begin{lemma}
	\label{lem:akkt_bounded}
	Let a feasible point $\bar x \in \Feas$ of \eqref{Eq:Opt} be given.
	\begin{enumerate}[(a)]
		\item\label{item:bounded_swCCP}
			Assume that for all sequences
			$\{x^k\}\subset X$, $\{r^k\}\subset [0,+\infty) $, and $\{v^k\} \subset X\dualspace$ 
			which satisfy
			$x^k \to \bar x$,
			$r^k \searrow 0$,
			$v^k \wtos \bar v$,
			and
			$v^k \in \MM(x^k, r^k)$ for all $k\in\N$,
			there exist bounded sequences of multipliers
			$\{\lambda^k\}\subset \SpaceY^*$ and $\{\mu^k\} \subset X\dualspace$
			such that
			$\mu^k \in \NN_C(x^k)$
			and
			$\dual{ \lambda^k }{ y - G(x^k) }_{\SpaceY} \leq r^k$ for all $y \in \ConY$
			and $k\in\N$
			as well as
			$v^k=G'(x^k)\adjoint \lambda^k + \mu^k \wtos \bar v$.
			Then $ \bar x $ is sw-AKKT regular.
		\item\label{item:bounded_wCCP}
			Assume that $G$ and $G'$ are completely continuous
			and $C = X$.
			Assume further that for all sequences
			$\{x^k\}\subset X$, $\{r^k\}\subset [0,+\infty) $, and $\{v^k\} \subset X\dualspace$ 
			which satisfy
			$x^k \wto \bar x$,
			$r^k \searrow 0$,
			$v^k \wtos \bar v$,
			and
			$v^k \in \MM(x^k, r^k)$ for all $k\in\N$,
			there exists a bounded sequence of multipliers
			$\{\lambda^k\}\subset \SpaceY^*$
			such that it holds
			$\dual{ \lambda^k }{ y - G(x^k) }_{\SpaceY} \leq r^k$ for all $y \in \ConY$
			and $k\in\N$
			as well as
			$v^k=G'(x^k)\adjoint \lambda^k\wtos \bar v$.
			Then $ \bar x $ is w-AKKT regular.
	\end{enumerate}
\end{lemma}

\begin{proof}
(a) Let $\bar v \in 
			w^*$-$\limsup_{
				\substack{x \to \bar x \\ r \searrow 0}
			} \M(x,r)
			$
			be given.
			Then
			there exist  sequences
			$\{x^k\}\subset X$, 
			$\{r^k\}\subset [0,+\infty) $, 
			and $\{v^k\} \subset X\dualspace$ such that
			$x^k \to \bar x$,
			$r^k \searrow 0$,
			$v^k \wtos \bar v$,
			and
			$v^k \in \MM(x^k, r^k)$ for all $k\in\N$.
			By assumption, 
			there exist bounded sequences
			$\{\lambda^k\}\subset \SpaceY^*$ and $\{\mu^k\} \subset X\dualspace$
			with
			$\mu^k \in \NN_C(x^k)$,
			$\dual{ \lambda^k }{ y - G(x^k) }_{\SpaceY} \leq r^k$ for all $y \in \ConY$ and $k\in\N$,
			and
			$v^k=G'(x^k)\adjoint \lambda^k + \mu^k \wtos \bar v$. By the
			Banach--Alaoglu--Bourbaki theorem, the sequences $\{\lambda^k \}$ and $\{\mu^k\}$ 
			possess weak* convergent subnets,
			indexed by $k(i)$, $i \in I$, where $I$ is a directed set.
			The associated weak* limits are denoted by $\lambda$ and $\mu$, respectively.
			We get
			\begin{align*}
				\forall y\in K\colon\quad
				\dual{\lambda}{y - G(\bar x)}_{\SpaceY}
				&\leftarrow
				\dual{\lambda^{k(i)}}{G(\bar x) - G(x^{k(i)})}_{\SpaceY}
				+
				\dual{\lambda^{k(i)}}{y - G(\bar x)}_{\SpaceY}
				\\
				&=
				\dual{\lambda^{k(i)}}{y - G(x^{k(i)})}_{\SpaceY}
				\le
				r^{k(i)}
				\to
				0.
			\end{align*}
			Note that the first limit uses the boundedness of the net $\{\lambda^{k(i)}\}$
			and this follows from the boundedness of the sequence $\{\lambda^k\}$.
			Thus, $\lambda \in \NN_K(G(\bar x))$ holds.
			Similarly,
			\begin{equation*}
				\forall x\in C\colon\quad
				\dual{\mu}{x - \bar x}_X
				=
				\lim\limits_{i \in I} \dual{\mu^{k(i)}}{x - x^{k(i)}}_X
				\le
				0
			\end{equation*}
			implies
			$\mu \in \NN_C(\bar x)$.
			Again, we used the boundedness of $\{\mu^{k(i)}\}$ which follows from the boundedness of $\{\mu^k\}$.
			Hence, it holds
			$G'(\bar x)\adjoint \lambda + \mu \in \MM(\bar x, 0)$.
				Using again the boundedness of $\{\lambda^{k(i)}\}$ and $G'(x^{k(i)}) \to G'(\bar x)$, we have
				\begin{equation*}
					\dual{ G'(x^{k(i)})\adjoint \lambda^{k(i)} }{x}_X
					=
					\dual{ \lambda^{k(i)} }{G'(x^{k(i)}) \, x}_Y
					\to
					\dual{ \lambda }{G'(\bar x) \, x}_Y
					=
					\dual{G'(\bar x) \adjoint \lambda }{ x}_Y
				\end{equation*}
				for all $x \in X$.
				This shows
				$G'(x^{k(i)})\adjoint \lambda^{k(i)}  \wtos G'(\bar x)\adjoint \lambda$.
			Due to the convergence 
			$v^{k(i)}=G'(x^{k(i)})\adjoint \lambda^{k(i)} + \mu^{k(i)} \wtos G'(\bar x)\adjoint \lambda + \mu$,
			we obtain from the uniqueness of the weak* limit point that
			$\bar v = G'(\bar x)\adjoint \lambda + \mu \in \MM(\bar x, 0)$,
			i.e., $\bar x $ is sw-AKKT regular.
	
(b) This follows by almost the same proof.
			Note that the complete continuities of $G$ and $G'$
			imply
			$G(x^{k(i)}) \to G(\bar x)$ and $G'(x^{k(i)}) \to G'(\bar x)$, respectively.
\end{proof}

Unfortunately,
the condition $C = X$ needed in the proof for the second statement regarding w-AKKT regularity 
is quite restrictive, but cannot be omitted as long as $X$ is infinite dimensional.
Otherwise, by reprising the above proof strategy,
we get bounded nets $\{x^{k(i)}\}\subset X$ and $\{\mu^{k(i)}\}\subset X^*$ satisfying
$x^{k(i)} \wto \bar x$, $\mu^{k(i)} \wtos \mu$, and $\mu^{k(i)} \in \NN_C(x^{k(i)})$ for all $i \in I$.
However, this is not enough to conclude
$\mu \in \NN_C(\bar x)$, see also \Cref{ex:RCQ_but_not_wCCCCP}.

\section{Relations to other Constraint Qualifications}\label{Sec:Relation}

As pointed out in the previous section,
weak, strong-weak, and strong AKKT regularity
are constraint qualifications for \eqref{Eq:Opt} under some additional assumptions on the
problem data. Therefore, the natural question arises how these new constraint
qualifications are related to already existing ones. The most prominent case is discussed in
\Cref{Sub:RCQ} where we show that Robinson's constraint qualification implies 
sw-AKKT regularity. 
As it will turn out, it even implies 
w-AKKT regularity 
under some additional assumptions. 
Afterwards, we consider the relationship between 
sw-AKKT regularity 
and Abadie's constraint qualification in \Cref{Sub:Abadie}. 
Finally, we conclude that 
s-AKKT regularity 
implies that at least Guignard's constraint
qualification holds at the reference point whenever $X$ is reflexive and separable
in \Cref{Sub:Guignard}.

\subsection{Robinson Constraint Qualification}\label{Sub:RCQ}

The most common constraint qualification in Banach spaces is Robinson's constraint
qualification which dates back to the seminal paper \cite{Robinson1976} where it
has been used to characterize variational stability of perturbed nonlinear systems
in Banach spaces. The interpretation of this condition as a constraint qualification
in Banach space programming is due to \cite{Zowe1979}.
The aim of this section is to show that Robinson's constraint qualification is stronger than 
sw-AKKT regularity. 
As we will see later in \Cref{Sub:Box}, it is strictly stronger than
sw-AKKT regularity 
in general.
To proceed, we first recall the definition of Robinson's constraint qualification.

\begin{definition}\label{Dfn:RobinsonCQ}
We say that \emph{Robinson's constraint qualification} (RCQ) holds at a feasible point $\bar x\in \Feas$ 
of \eqref{Eq:Opt} if
the condition
\begin{equation}
	\label{eq:rzkcq}
	Y
	=
	G'(\bar x) \, \RR_C(\bar x) - \RR_K(G(\bar x))
\end{equation}
is valid.
\end{definition}

In the theorem below, we show that validity of RCQ at some reference point always guarantees 
sw-AKKT regularity. 

\begin{theorem}
	\label{lem:RCQ_implies_sw_CCP}
	Assume that RCQ is satisfied at a feasible point $\bar x \in \Feas$ of \eqref{Eq:Opt}.
	Then $ \bar x $ is sw-AKKT regular (and, thus, s-AKKT regular). 
\end{theorem}

\begin{proof}
	We check that the assumption of \Cref{lem:akkt_bounded}~\ref{item:bounded_swCCP}
	is satisfied.
	To this end,
	let sequences
	$\{x^k\}\subset X$, $\{r^k\}\subset [0,+\infty) $, and $\{ v^k \} \subset X\dualspace$ with
	$x^k \to \bar x$,
	$r^k \searrow 0$,
	$v^k \wtos \bar v$,
	and
	$v^k \in \MM(x^k, r^k)$
	for all $k\in\N$
	be given.
	By definition of $\MM(x^k, r^k)$,
	this implies the existence of sequences
	$\{\lambda^k\}\subset\SpaceY^*$ and $\{ \mu^k \} \subset X\dualspace$
	with
	$\mu^k \in \NN_C(x^k)$ and
	$\dual{ \lambda^k }{ y - G(x^k) }_{\SpaceY} \leq r^k$ for all $y \in \ConY$
	and $k\in\N$, as well as
	$G'(x^k)\adjoint \lambda^k + \mu^k = v^k \wtos \bar v$.
	It suffices to verify the boundedness of $\{\lambda^k\}$ and $\{\mu^k\}$.

	Under assumption \eqref{eq:rzkcq},
	we can apply the generalized open mapping theorem
	\cite[Thm.\ 2.1]{Zowe1979}
	and obtain the existence of $M > 0$,
	such that for all $z \in Y$ with $\norm{z}_{\SpaceY}\le1$,
	there exist $w \in C\cap B_1(\bar x)$ 
	and $y \in K\cap B_1(G(\bar x))$
	such that
	\begin{equation*}
		-\frac{z}{M} = G'(\bar x) \, (w - \bar x) - (y - G(\bar x)).
	\end{equation*}
	We fix an arbitrary $z \in Y$ with $\norm{z}_{\SpaceY} \le 1$
	and the corresponding vectors $w$ and $y$ from above. 
	Then let us write
	\begin{equation*}
		-\frac{z}{M}
		=
		G'(x^k) \, (w - x^k) - (y - G(x^k)) + \delta^k
	\end{equation*}
	with
	$\delta^k := G'(\bar x) \, (w - \bar x) - G'(x^k) \, (w - x^k) + G(\bar x) - G(x^k)$.
	We have the estimate
	\begin{align*}
		\norm{ \delta^k }_{\SpaceY}
		&\le
		\norm{G'(\bar x) \, (x^k - \bar x) }_{\SpaceY}
		+
		\norm{G'(\bar x) - G'(x^k)} \, \norm{w - x^k}_X
		+
		\norm{G(\bar x) - G(x^k)}_{\SpaceY}
		\\
		&\le
		\norm{G'(\bar x) \, (x^k - \bar x) }_{\SpaceY}
		+
		\norm{G'(\bar x) - G'(x^k)} \, (1 + \norm{\bar x - x^k}_X)
		+
		\norm{G(\bar x) - G(x^k)}_{\SpaceY}
		\\&=:
		s^k,
	\end{align*}
	where $s^k \to 0$ holds by continuity of $ G $ and $ G' $.
	Note that $s^k$ is independent of $z$.
	We have
	\begin{align*}
		\dual[\Big]{\lambda^k}{\frac{z}{M}}_{\SpaceY}
		&=
		-\dual{G'(x^k)^*\lambda^k}{w - x^k}_{X} + \dual{\lambda^k}{y - G(x^k)}_{\SpaceY} - \dual{\lambda^k}{\delta^k}_{\SpaceY}
		\\
		&\le
		\dual{\mu^k - v^k}{w - x^k}_{X} + r^k + \norm{\delta^k}_{\SpaceY}\, \norm{\lambda^k}_{\SpaceY^*}
		\\
		&\le
		\norm{v^k}_{X^*} \, (1 + \norm{\bar x - x^k}_X) + r^k + s^k \, \norm{\lambda^k}_{\SpaceY^*}.
	\end{align*}
	In the last inequality, we used
	$\dual{\mu^k}{w - x^k}_{X} \le 0$ due to $\mu^k \in \NN_C(x^k)$ and $w \in C$.
	Since $z \in Y$ with $\norm{z}_{\SpaceY} \le 1$ was arbitrary
	and since the right-hand side in the above inequality is independent of $z$, this implies
	\begin{equation*}
		\norm{\lambda^k}_{\SpaceY^*}
		\le
		M \, \paren[\big]{ \norm{v^k}_{X^*} \, (1 + \norm{\bar x - x^k}_X) + r^k + s^k \, \norm{\lambda^k}_{\SpaceY^*}}.
	\end{equation*}
	Due to $s^k \to 0$, we can conclude
	\begin{equation*}
		\norm{\lambda^k}_{\SpaceY^*}
		\le
		2 \, M \, \paren{ \norm{v^k}_{X^*} \, (1 + \norm{\bar x - x^k}_X) + r^k }
	\end{equation*}
	for all large enough $k\in\N$.
	Since $\{v^k\}$ and $\{x^k\}$ are bounded,
	the boundedness of $\{\lambda^k\}$ follows.
	Finally, due to $\mu^k = v^k - G'(x^k)\adjoint \, \lambda^k$, $\{\mu^k\}$ is bounded as well.
\end{proof}

In the theorem below, we investigate situations where validity of RCQ already implies
w-AKKT regularity. 

\begin{theorem}
	\label{lem:RCQ_implies_w_CCP}
	Assume that RCQ is satisfied at a feasible point $\bar x \in \Feas$ of \eqref{Eq:Opt}.
	Further, assume $C = X$
	and that $G$ and $G'$ are completely continuous.
	Then $ \bar x $ is w-AKKT regular. 
\end{theorem}

\begin{proof}
	For the verification of \Cref{lem:akkt_bounded}~\ref{item:bounded_wCCP},
	we can transfer the proof of \Cref{lem:RCQ_implies_sw_CCP} to the 
	situation at hand.
	We have to use the boundedness of weakly and weak* convergent
	sequences.
	Moreover,
	the complete continuity of $G$ and $G'$
	as well as the compactness of $G'(\bar x)$
	have to be used to conclude $s^k \to 0$,
	where $s^k$ is defined as above.
\end{proof}

By means of an example, we show that
RCQ does not imply 
w-AKKT regularity 
in general.

\begin{example}
	\label{ex:RCQ_but_not_wCCCCP}
	We consider $X := \ell^2$
	and its unit ball
	$C := \set{ x \in \ell^2 \given \norm{x}_{\ell^2} \le 1}$.
	We identify $X\dualspace$ with $X$.
	Furthermore, we assume that no further constraints are present,
	i.e., $Y = K = \{0\}$ and $G \colon X \to Y$ is the zero mapping.
	Then $\bar x := e^1 / \sqrt{2}$ is an interior point of $C$ and, consequently,
	\eqref{eq:rzkcq} is satisfied at $\bar x$.
	We define the sequence $\{x^k\}\subset C$ by means of
	$x^k := (e^1 + e^{k+1}) / \sqrt{2}$ for each $k\in\N$.
	Above, $e^n\in\ell^2$ denotes the $n$-th unit sequence in $\ell^2$
	for each $n\in\N$.
	For $k \in \N$, it is easy to see that we have $x^k \in \NN_C(x^k)$.
	In particular, we obtain
	$x^k \in \MM(x^k, 0)$ for all $k\in\N$.
	From $x^k \weakly \bar x$, we infer
	\begin{equation*}
		\bar x \in 
		\wStarlimsup_{
			\substack{x \wto \bar x \\ r \searrow 0}	
		} \M(x,r),
	\end{equation*}
	but
	$\bar x \not\in \MM(\bar x, 0) = \Nor_C( \bar x) = \{0\}$, where
	the first equality holds since the feasible set is defined by abstract
	constraints only, whereas the second equality exploits the fact that
	$ \bar x $ is an interior point of $ C $.
	Hence, 
	$\bar x$ is not w-AKKT regular. 
	\hfill$\Diamond$
\end{example}
It is clear that similar examples can be constructed
in all infinite-dimensional Hilbert spaces.

\subsection{Abadie Constraint Qualification}\label{Sub:Abadie}

In the finite-dimensional setting, it was shown in \cite{AndreaniMartinezRamosSilva2016} 
that in case of its presence, AKKT regularity implies 
validity of Abadie's constraint qualification
which originates from \cite{Abadie1965}.
Here, we want to generalize this observation to the infinite-dimensional situation.
Let us first state an appropriate notion of Abadie's constraint qualification which
applies to the general situation discussed in this paper.

\begin{definition}
	Let $\bar x\in\Feas$ be a feasible point of \eqref{Eq:Opt}.
	We say that \emph{Abadie's constraint qualification} (ACQ) is valid at $\bar x$ if
	\begin{equation}
		\label{eq:sACQ}
		\Tng_\Feas(\bar x) = \mathcal{L}_{\Feas}(\bar x)
	\end{equation}
	holds and $\M(\bar x,0)$ is weak* closed.
\end{definition}

Recall that $\mathcal L_\Feas(\bar x)$ from \eqref{eq:linearization_cone}
denotes the linearization cone to $\Feas$ at $\bar x$, and
that $\MM(\bar x,0)$ can be used to characterize the KKT conditions,
see \eqref{Eq:SeqMbarx0}, which shows that the relation
\begin{equation*}
   \M(\bar x,0) = G'(\bar x)^* \Nor_K( G(\bar x)) + \Nor_C (\bar x)
\end{equation*}
holds at the given feasible point $ \bar x \in\Feas $.
We note that demanding $\M(\bar x,0)$ to be weakly* closed is, in general, indispensable
in the definition of ACQ in order to guarantee that it is a constraint qualification in
the narrower sense. Indeed, the polarity relation 
\begin{equation}\label{eq:Farkas}
   \mathcal L_\Feas(\bar x)^\circ = G'(\bar x)^* \Nor_K( G(\bar x)) + \Nor_C (\bar x) = \M(\bar x,0),
\end{equation}
which comes for free in the context of standard finite-dimensional nonlinear programming,
only holds if $\M(\bar x,0)$ is weakly* closed, see \cite[Sec.\ 2]{Kurcyusz1976}
or \cite[Lem.\ 1]{Zalinescu1978}. Observe that in case where $X$ is reflexive, the closedness
of $\M(\bar x,0)$ already yields its weak* closedness, see \Cref{rem:closedness_via_CCP} as well.
The above arguments yield the following well-known result which is
included for the reader's convenience.

\begin{proposition}\label{lem:ACQ_is_a_CQ}
	Let $\bar x\in\Feas$ be a local minimizer of \eqref{Eq:Opt}
	where ACQ holds. Then $\bar x$ is a KKT point of \eqref{Eq:Opt}.
\end{proposition}

\begin{proof}
	Since $\bar x\in\Feas$ is a local minimizer of \eqref{Eq:Opt}, 
	it holds $f'(\bar x)d\geq 0$ for all directions $d\in\TT^w_\Feas(\bar x)$, i.e.,
	$-f'(\bar x)\in\TT^w_\Feas(\bar x)^\circ$. 
	Recalling that 
	$\TT_\Feas(\bar x)\subset\TT_\Feas^w(\bar x)\subset\mathcal L_\Feas(\bar x)$
	holds in general, the validity of ACQ guarantees 
	$\TT^w_\Feas(\bar x)=\mathcal L_\Feas(\bar x)$.
	Thus, the above arguments show $\TT^w_\Feas(\bar x)^\circ=\M(\bar x,0)$,
	i.e., $-f'(\bar x)\in\M(\bar x,0)$ follows.
	Hence, $\bar x$ is a KKT point of \eqref{Eq:Opt}.
\end{proof}

From this proof it is clear that \eqref{eq:sACQ}
could be weakened to $\TT_\Feas^w(\bar x) = \LL_\Feas(\bar x)$
in the definition of ACQ.

In order to prove the main result of this section, we need two technical preliminaries which will be
provided below.

\begin{lemma}\label{Lem:eqivFdiffNorm}
Suppose that $ X^* $ is separable. Then there is an equivalent norm
on $ X $
which is continuously Fréchet differentiable in $ X \setminus \{ 0 \} $.
\end{lemma}

\begin{proof}
Combine the results from \cite[Cor.\ 8.5 and Thm.\ 8.19]{Fabian2001}.
\end{proof}

In the remaining parts of this section,
we will assume that the space $ X $ is equipped with the norm from the
above lemma.
Recall that reflexivity and separability of $X$ together
imply separability of $X\dualspace$.
\begin{lemma}
	\label{lem:some_function}
	Suppose that $ X $ is reflexive and separable.
	Fix a feasible point $\bar x\in\Feas$ of \eqref{Eq:Opt}.
	For every $ v \in \widehat\NN_\Feas(\bar x)$ there
	is a continuously Fréchet differentiable function $ h\colon X \to \R $
	with $h'(\bar x) = -v $
	such that $ h $
	restricted to $\Feas$
	achieves a unique
	global minimum at $ \bar x $. Moreover,
	the function $h$ is weakly sequentially lower semicontinuous.
\end{lemma}
\begin{proof}
	It is well known that every differentiable convex function is already 
	continuously differentiable, c.f.\ \cite[Cor.\ of Prop.~2.8]{Phelps1993}.
	The above assertion therefore follows directly from \cite[Thm.\ 1.30(ii)]{Mordukhovich2006}, 
	since $ X $ admits a continuous differentiable renorm by \Cref{Lem:eqivFdiffNorm}.
\end{proof}

Now, we can transfer the proof of \cite[Thm.\ 4.4]{AndreaniMartinezRamosSilva2016}
to the infinite-dimensional setting.

\begin{theorem}
	\label{thm:ccp_implies_abadie}
	Let $X$ be reflexive and separable.
	Let us assume that $ \bar x $ is an sw-AKKT regular feasible point $\bar x \in \Feas$
	of \eqref{Eq:Opt}.
	Furthermore, assume that condition \eqref{eq:continuity_property_distance_constraint_map}
	holds.
	Then ACQ is valid at $\bar x$.
\end{theorem}

\begin{proof}
	In a preliminary step, we first verify the inclusion 
	$\NN^\textup{L}_\Feas(\bar x) \subset \MM(\bar x, 0)$.
	Let $v \in \NN^\textup{L}_\Feas(\bar x)$ be given.
	Then, due to reflexivity of $X$, there exist sequences $\{x^k\}\subset\Feas$ and $\{v^k\}\subset X^*$
	such that $x^k\to\bar x$, $v^k\wto v$, and $v^k\in\widehat{\NN}_\Feas(x^k)$ for all $k\in\N$.
	Invoking \Cref{lem:some_function} for each $k\in\N$,
	there exists a function $h^k\colon X\to\R$
	such that $x^k$ is the constrained minimizer of $h^k$ over the feasible set $\Feas$
	and $(h^k)'(x^k)=-v^k$.
	An inspection of the corresponding proof shows that \Cref{Prop:OptAKKT2} guarantees the
	existence of  sequences $\{x^{k,\ell}\}\subset C$ and $\{v^{k,\ell}\}\subset X^*$ such that
	\begin{equation*}
		x^{k,\ell} \to x^k,
		\qquad
		v^{k,\ell} \to v^k,
		\qquad
		v^{k,\ell} \in \MM(x^{k,\ell}, 0)\quad\forall\ell\in\N
		.
	\end{equation*}
	Thus, we can pick diagonal sequences $\{x^{k,\ell(k)}\}$ and $\{v^{k,\ell(k)}\}$ with
	\begin{equation*}
		x^{k,\ell(k)} \to \bar x,
		\quad
		v^{k,\ell(k)} \weakly v
		.
	\end{equation*}
	Naturally, we have $v^{k,\ell(k)}\in\M(x^{k,\ell(k)},0)$ for each $k\in\N$.
	Now, 
	sw-AKKT regularity 
	yields $v \in \MM(\bar x, 0)$.
	This shows $\NN^\textup{L}_\Feas(\bar x) \subset \MM(\bar x,0)$.

	To verify the statement of the theorem, note that \Cref{rem:closedness_via_CCP} and 
	the sw-AKKT regularity 
	imply the weak* closedness of $\MM(\bar x, 0)$ since $X$ is assumed to be reflexive.
	As pointed out above, this yields the polar relationship
	$\MM(\bar x,0) = \LL_\Feas(\bar x)\polar$, cf.\ \eqref{eq:Farkas}.
	Using our preliminary step, we therefore have $ \NN^\textup{L}_\Feas(\bar x) \subset 
	\MM(\bar x, 0) = \LL_\Feas(\bar x)\polar$. Taking polars yields
	$\NN^\textup{L}_\Feas(\bar x)\polar \supset \LL_\Feas(\bar x)$.
	Furthermore, \cite[Thm.\ 3.57]{Mordukhovich2006} guarantees
	$\NN^\textup{L}_\Feas(\bar x)\polar = \TT^{\textup{C}}_\Feas(\bar x)$.
	Hence, we get the chain of inclusions
	\begin{equation*}
		\LL_\Feas(\bar x)
		\subset
		\TT^{\textup{C}}_\Feas(\bar x)
		\subset
		\TT_\Feas(\bar x)
		\subset
		\TT_\Feas^w(\bar x)
		\subset
		\LL_\Feas(\bar x),		
	\end{equation*}
	and this finishes the proof.
\end{proof}

The following corollary can be distilled from the proof of \Cref{thm:ccp_implies_abadie}.
\begin{corollary}\label{cor:equality_of_tangent_cones_under_sw_CCP}
	Let $X$ be reflexive and separable.
	Let $\bar x\in\Feas$ be a feasible sw-AKKT regular point of \eqref{Eq:Opt}.
	Finally, assume that condition
	\eqref{eq:continuity_property_distance_constraint_map} holds.
	Then we have the equalities
	\[
		\LL_\Feas(\bar x)
		=
		\TT^\textup{C}_\Feas(\bar x)
		=
		\TT_\Feas(\bar x)
		=
		\TT^w_\Feas(\bar x).
	\]
\end{corollary}

\subsection{Guignard Constraint Qualification}\label{Sub:Guignard}

Let us first recall the definition of Guignard's constraint qualification which
can be traced back to \cite{Guignard1969}.

\begin{definition}\label{def:GCQ}
Let $\bar x\in\Feas$ be a feasible point of \eqref{Eq:Opt}.
We say that \emph{Guignard's constraint qualification} (GCQ) is valid at $\bar x$ if
\begin{equation*}
	\TT_\Feas^w(\bar x)\polar
	=
	\MM(\bar x, 0).
\end{equation*}
holds.
\end{definition}

Recall that our definition of GCQ is not necessarily equivalent to the requirement 
$ \TT_\Feas^w(\bar x)\polar = \LL_\Feas(\bar x)\polar $. In fact, this is true only
if $ \MM(\bar x, 0) $ is weak* closed, cf.\ \eqref{eq:Farkas}. By polarity, 
\Cref{def:GCQ} immediately implies that $ \MM(\bar x, 0) $ is weak* closed, whereas 
this does not follow from the alternative definition that 
$ \TT_\Feas^w(\bar x)\polar = \LL_\Feas(\bar x)\polar $.

Inspecting the proof of \Cref{lem:ACQ_is_a_CQ}, it is clear that GCQ is indeed a
constraint qualification.
Furthermore, under the assumption that $X$ is reflexive,
GCQ is the weakest constraint qualification
which ensures
that $\bar x$ is a KKT point for all functions $\hat f$
which are differentiable at $\bar x$
and for which
$\bar x$ is a local minimizer of $\hat f$ over $\Feas$,
see \cite[Cor.\ 3.4]{GouldTolle1975}.
We show that 
s-AKKT regularity 
implies GCQ if $X$ is additionally separable.

\begin{theorem}
	\label{lem:sCCP_implies_GCQ}
	Let $X$ be reflexive and separable. Let us assume that 
	$\bar x\in\Feas$ is a feasible s-AKKT regular point of \eqref{Eq:Opt}. 
	Furthermore, assume that condition \eqref{eq:continuity_property_distance_constraint_map} holds.
	Then GCQ is valid at $\bar x$.
\end{theorem}

\begin{proof}
	From $\TT_\Feas^w(\bar x) \subset \LL_\Feas(\bar x)$
	we obtain $\TT_\Feas^w(\bar x)\polar \supset \LL_\Feas(\bar x)\polar = \MM(\bar x, 0)$
	from taking polars
	since $\MM(\bar x, 0)$ is closed 
	by the postulated s-AKKT regularity, 
	cf.\ \Cref{rem:closedness_via_CCP}.
	For $v \in \TT_\Feas^w(\bar x)\polar = \widehat\NN_\Feas(\bar x)$,
	we can argue similarly to the proof of \Cref{thm:ccp_implies_abadie}:
	\Cref{lem:some_function} yields a
	continuously differentiable function $ h: X \to \R $ such that $ h $
	restricted to $\Feas$
	achieves a unique
	global minimum at $ \bar x $ and $h'(\bar x) = -v $.
	Now, we can apply \Cref{Cor:AKKT-CQ_is_CQ} to the objective $h$
	and obtain $v = -h'(\bar x) \in \MM(\bar x, 0)$.
	This finishes the proof.
\end{proof}

In \Cref{fig:CQs}, we summarize the relations between all 
types of AKKT regularity 
from \Cref{Def:AKKTCQ} as well
as RCQ, ACQ, and GCQ in the context of a reflexive Banach space $X$. 
In light of \Cref{Prop:OptAKKT2} and \Cref{Ex:necessity_of_reflexivity}, the reflexivity assumption 
on $X$ is indispensable whenever sequential constraint qualifications are under consideration.

\begin{figure}[h]
\centering
\begin{tikzpicture}[->]

  \node[punkt] at (0,0) 	(A){RCQ};
  \node[punkt] at (-3.5,-2) 	(B){w-AKKT reg.};
  \node[punkt] at (0,-2) 	(C){sw-AKKT reg.};
  \node[punkt] at (3.5,-2)    (D){s-AKKT reg.};
  \node[punkt] at (0,-4) 	(E){ACQ};
  \node[punkt] at (3.5,-4)  	(F){GCQ};

  \path     (A) edge[-implies,thick,double] node[above left] {(a)}(B)
            (A) edge[-implies,thick,double] node {}(C)
            (B) edge[-implies,thick,double] node {}(C)
            (C) edge[-implies,thick,double] node {}(D)
            (C) edge[-implies,thick,double] node[left] {(b)}(E)
            (D) edge[-implies,thick,double] node[left] {(b)}(F)
            (E) edge[-implies,thick,double] node {}(F);
\end{tikzpicture}
\caption{
	Relations between constraint qualifications in the setting where $X$ is reflexive.
	Relations with labeled arrows only hold under additional assumptions:\\
	(a) requires complete continuity of $G$ and $G'$ as well as $C=X$,
	(b) holds whenever $X$ is separable and \eqref{eq:continuity_property_distance_constraint_map}
	holds.
}
\label{fig:CQs}
\end{figure}
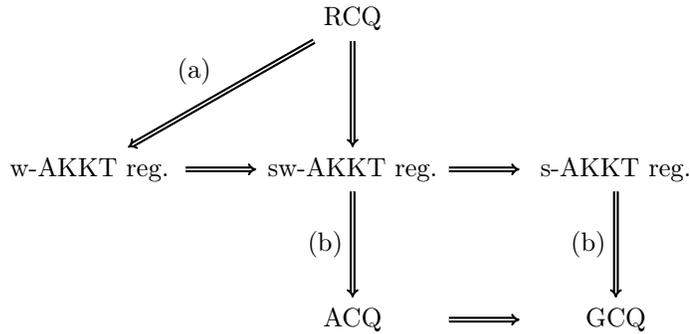

We close this section by pointing out one open problem, namely whether
s-AKKT regularity 
already implies ACQ.
The previous technique of proof does not yield this implication, on the other hand,
we were also not able to find a counterexample.

\section{AKKT Regularity in Exemplary Settings}\label{sec:CCP_in_practice}

In this section, we present three practically relevant settings where w-AKKT regularity or at least sw-AKKT regularity 
is inherent or can be checked via evaluation of reasonable conditions.
First, we prove in \Cref{Sub:Affine} that 
w-AKKT regularity 
is automatically satisfied 
in the setting where the feasible set $\Feas$ is defined 
via affine equality constraints which are induced by a bounded linear operator with
a closed range. \Cref{Sub:Nonlinear} discusses the 
situation where we have nonlinear equality constraints which model $\Feas$.
Finally, in \Cref{Sub:Box}, we will investigate the important 
setting of two-sided (pointwise) box constraints in Lebesgue spaces.
It will be shown that 
sw-AKKT regularity 
holds in this situation as well (whereas RCQ is known to be violated 
for this class of problems).

\subsection{Affine Equality Constraints}\label{Sub:Affine}

Let $X$ and $Y$ be Banach spaces such that $X$ is reflexive.
Furthermore, fix an operator $A\in\mathbb L(X,Y)$ with closed range and some vector $b\in Y$. 
We consider the affine equality constraint
\[
	Ax=b.
\]
In this situation, it holds $\ConY:=\{0\}$, $C:=X$, and $G(x):=Ax-b$ for all $x\in X$.
Clearly, RCQ holds for this constraint system if and only if $A$ is surjective, i.e.,
if $\range A=\SpaceY$.

By $K=\{0\}$, it clearly holds
\[
	\forall x\in X\,~\forall r\in [0,+\infty) \colon\quad
	\M(x,r)
	=
	\{
		A^*\lambda
		\,|\,
		\lambda\in Y^*,\,-\<\lambda,Ax-b\>_Y\leq r
	\},
\]
i.e., for each point $\bar x\in X$ satisfying $A\bar x=b$, we obtain
\[
	\forall r\in [0,+\infty) \colon\quad
	\M(\bar x,r)
	= A^*Y^*
\]
Particularly, $\M(x,r) \subset \M(\bar x,0)$ is obtained for arbitrary $x\in X$ and $r\in [0,+\infty) $.
Obviously, $\M(\bar x,0)=A^*\SpaceY^*$ is convex.
Furthermore,  $\M(\bar x,0)$ is closed by closedness of $AX$ and the closed range theorem.
Since $X$ is reflexive, the same holds true for $X^*$ and, consequently, $\M(\bar x,0)$ is
a weakly* closed subset of $X^*$.
Particularly, we obtain 
w-AKKT regularity of each point  
which satisfies the affine constraint.

In view of \Cref{Prop:OptAKKT2} and \Cref{sAKKTpIsKKT}, the above observation recovers the classical
result \cite[Thm.\ 4.2]{Kurcyusz1976}.

\subsection{Nonlinear Equality Constraints}\label{Sub:Nonlinear}

We consider the special constraint system
\begin{equation}\label{eq:nonlinear_inequality}
	G(x)=0,
\end{equation}
i.e., we fix $C: = X$ and $K := \{0\}$.
In contrast to \Cref{Sub:Affine}, this equality constraint
is allowed to be nonlinear.
Clearly, RCQ holds at a feasible point $\bar x\in X$ of the constraint system
\eqref{eq:nonlinear_inequality} if and only if $G'(\bar x)$ is surjective.

The next result utilizes the
\emph{reduced minimum modulus}
introduced in \cite[Sec.\ IV. § 5]{Kato1995}.
For a bounded linear operator $T \in\mathbb L(X,Y)$,
it is defined via
\begin{equation*}
	\gamma(T)
	:=
	\inf\set[\big]{
		\norm{T \, x}_Y
		\given
		x \in X
		,\;
		\dist(x, \ker T)
		=
		1
	}
	.
\end{equation*}
It is well known that
the range of $T$ is closed
if and only if
$\gamma(T) > 0$,
see \cite[Thm.\ IV.5.2]{Kato1995}.
Hence,
$\gamma(T)$
can be used as a
``quantitative measure of closedness''
of the range of $T$.
Moreover,
$\gamma(T) = \gamma(T\adjoint)$
holds
and this is a quantitative version of the closed range theorem,
see \cite[Thm.\ IV.5.13]{Kato1995}.
Finally,
we have the inequality
\begin{equation}
	\label{eq:pseudo-inverse-inequality}
	\forall x\in X\colon\quad
	\dist(x, \ker T)
	=
	\inf_{v \in \ker T} \norm{x - v}_X
	\le
	\frac1{\gamma(T)} \,
	\norm{T \, x}_Y
\end{equation}
in case $\gamma(T) > 0$.
In finite dimensions,
$\gamma(T)$ coincides with the
reciprocal of the norm of the Moore--Penrose inverse of the matrix $T$.

\begin{proposition}
	\label{prop:ccp_constant_rank}
	Let $\bar x\in \Feas$ be a feasible point of the constraint system \eqref{eq:nonlinear_inequality}.
	Furthermore, suppose that there are a neighborhood $U\subset X$ of $\bar x$ and some $\beta>0$
	such that 
	\begin{equation}\label{eq:CRCQ}
		\forall x\in U\colon\quad
		\gamma( G'(x) ) \ge \beta
	\end{equation}
	is valid.
	Then $ \bar x $ is sw-AKKT regular. 
\end{proposition}

\begin{proof}
	For the proof, we are going to exploit similar arguments as used in
	the validation of \Cref{lem:akkt_bounded}.
	Therefore, we fix sequences $\{x^k\}\subset X$, 
	$\{r^k\}\subset [0,+\infty) $, and
	$\{v^k\}\subset X^*$ with $x^k\to\bar x$, $r^k\searrow 0$, $v^k\wtos\bar v$,
	and $v^k\in\M(x^k,r^k)$ for all $k\in\N$. 
	Due to $K=\{0\}$, we find a sequence $\{\lambda^k\}\subset\SpaceY^*$
	such that $v^k=G'(x^k)^*\lambda^k$ and $-\dual{\lambda^k}{G(x^k)}_{\SpaceY}\leq r^k$
	hold for all $k\in\N$.
	By $v^k\wtos\bar v$, $\{G'(x^k)^*\lambda^k\}$ is bounded.
	Due to the assumptions of the proposition and \eqref{eq:pseudo-inverse-inequality},
	we find some constant $c>0$ such that
	\begin{equation*}
		\inf_{\mu^k \in \ker G'(x^k)\adjoint} \norm{ \lambda^k - \mu^k }_{Y\dualspace}
		\le
		\frac1\beta \, \norm{ G'(x^k)\adjoint \lambda^k }_{X\dualspace}
		\le
		c
	\end{equation*}
	for large enough $k\in\N$ since $x^k\to\bar x$ holds.
	Thus, we find a bounded sequence $\{\hat\lambda^k\} \subset Y\dualspace$
	with $\lambda^k - \hat\lambda^k \in \ker G'(x^k)\adjoint$,
	i.e., $v^k=G'(x^k)^*\hat\lambda^k$ for all $k\in\N$.
	Due to the boundedness,
	we obtain a subnet with
	$\hat\lambda^{k(i)}\wtos\hat\lambda$ for some $\hat\lambda\in\SpaceY^*$.
	Exploiting the continuity of $G'$ and the boundedness of $\{\hat\lambda^{k(i)}\}$,
	$G'(x^{k(i)})^*\hat\lambda^{k(i)}\wtos G'(\bar x)^*\hat\lambda$ holds,
		see the proof of \cref{lem:akkt_bounded}~(a).
	The uniqueness of weak* limits yields $G'(\bar x)^*\hat\lambda=\bar v$,
	and this shows $\bar v\in\M(\bar x,0)$.
	Hence, we have shown 
	$\wStarlimsup_{\substack{x \to \bar x \\ r \to 0}}\M(x,r)\subset\M( \bar x , 0 )$,
	i.e., $ \bar x $ is sw-AKKT regular. 
\end{proof}

Unfortunately,
the mapping $T \mapsto \gamma(T)$ is
in general
only upper semicontinuous,
e.g.,
for
\begin{equation*}
	T_\varepsilon
	:=
	\begin{pmatrix}
		1 & 0 \\ 0 & \varepsilon
	\end{pmatrix}
	\in \R^{2 \times 2}
\end{equation*}
we have
$\gamma(T_0) = 1$
but
$\gamma(T_\varepsilon) = \varepsilon^{-1}$ for $\varepsilon > 0$.
In finite dimensions,
$\gamma$ is continuous
on the set of matrices with constant rank,
and this justifies that \eqref{eq:CRCQ} can be 
interpreted as some kind of 
\emph{constant rank constraint qualification}.
In this regard, let us mention that
\Cref{prop:ccp_constant_rank}
generalizes the results of
\cite{Blot2018} since we do not need that
the kernel and range of $G'(\bar x)$
are complemented.

In infinite dimensions,
conditions ensuring continuity of $\gamma$
can be found in
\cite[Sec.\ 3]{ZhuCaiChen2003}
and
\cite[Sec.\ 6.1]{Xue2012}.
To formulate these conditions,
we introduce the gap between subspaces $U, V \subset X$
via
\begin{equation*}
	\delta( U, V )
	:=
	\sup \set[\big]{\dist(x, V) \given x \in U, \norm{x}_X = 1 },
\end{equation*}
with the convention $\delta(\{0\}, V ) = 0$.
We have
\begin{equation*}
	\gamma( T^k )
	\ge
	\gamma( T )
	\frac{1 - \delta(\ker T, \ker T^k)}{1 + \delta(\ker T, \ker T^ n)}
	-
	\norm{T - T^k}
\end{equation*}
for operators $T,T^k\in\mathbb L(X,Y)$.
If the range of $T^k$ is closed, we also have
\begin{equation*}
	\gamma( T^k )
	\ge
	\gamma( T )
	\frac{1 - \delta(\range T^k, \range T)}{1 + \delta(\range T^k, \range T)}
	-
	\norm{T - T^k}
	.
\end{equation*}
These two estimates can be found in
\cite[Prop.\ 6.1.5]{Xue2012}.
The first estimate is in particular applicable if $\ker T = \ker T^k$
since this implies $\delta( \ker T, \ker T^k ) = 0$.
Further, we have the trivial implications
\[
	\begin{aligned}
	&\ker T = \{0\}& &\Longrightarrow& &\delta( \ker T, \ker T^k ) = 0,& \\
	&\range T = Y&  &\Longrightarrow& &\delta( \range T^k, \range T ) = 0,&
	\end{aligned}
\]
and under any of the left-hand side requirements,
we have
$\gamma(T^k) \ge \gamma(T) - \norm{T - T^k}$.
Moreover,
it holds
\begin{equation*}
	\abs{ \gamma(T) - \gamma(T^k) } \le \norm{T - T^k}
\end{equation*}
if $T$ has closed range and if any of the following assumptions is satisfied
\begin{enumerate}
	\item
		$\dim\ker T = \dim\ker T^k < \infty$,
	\item
		$\dim\range T = \dim\range T^k < \infty$,
	\item
		$\range T^k$ is closed
		and
		$\dim\coker T = \dim\coker T^k < \infty$,
\end{enumerate}
see
\cite[Prop.\ 6.1.6]{Xue2012}.

\subsection{Box Constraints in Lebesgue Spaces}\label{Sub:Box}

Let $\Omega\subset\R^d$ be a bounded open set.
For functions $u_a,u_b\in L^2(\Omega)$ such that
$u_a\leq u_b$ holds almost everywhere on $\Omega$, 
we consider the box constraints
\[
	u_a\leq u\leq u_b\quad\text{a.e.\ on }\Omega.
\]
Here $u$ is a function from $L^2(\Omega)$.
Using $X:=L^2(\Omega)$, $Y:=L^2(\Omega)^2$, 
\[
	\ConY:=\{(y_1,y_2)\in L^2(\Omega)^2\,|\,y_1,y_2\geq 0\text{ a.e.\ on }\Omega\},
\]
and
$G(u):=(u_b-u,u-u_a)$ for each $u\in L^2(\Omega)$, the above box constraints can be described
equivalently via $G(u)\in\ConY$. 
For completeness, let us set $C:=L^2(\Omega)$.
It is easy to see that RCQ does not hold in this setting, see \cite[Section~6.1.2]{Troeltzsch2010}.
However, we will show below that all  associated feasible points
are sw-AKKT regular. 
To this end, we will make use of the strategy proposed in \Cref{lem:akkt_bounded}.

The associated set-valued mapping $\M\colon L^2(\Omega)\times[0,+\infty)\rightrightarrows L^2(\Omega)$
is given by
\[
	\M(u,r)
	=
	\left\{
		\lambda_a-\lambda_b
		\,\middle|\,
		\begin{aligned}
			&\lambda_a,\lambda_b\in L^2(\Omega),\;\lambda_a,\lambda_b\leq 0\text{ a.e.\ on }\Omega,\\
			&-\<\lambda_b,u_b-u\>_{L^2(\Omega)}-\<\lambda_a,u-u_a\>_{L^2(\Omega)}\leq r		
		\end{aligned}
	\right\}
\]
for all $u \in L^2(\Omega)$ and 
$r \in [0,+\infty) $,
see \Cref{rem:K_cone}. 
Fix some point $\bar u\in L^2(\Omega)$ satisfying $G(\bar u)\in K$.
Furthermore, let $\{u^k\}\subset L^2(\Omega)$, 
$\{r^k\}\subset [0,+\infty) $,
and $\{\lambda^k\}\subset L^2(\Omega)$ be sequences such that
$\lambda^k\in\M(u^k,r^k)$ for each $k\in\N$ as well as $u^k\to\bar u$ in $L^2(\Omega)$,
$r^k\searrow 0$
in $\R$, and $\lambda^k\wto\bar \lambda$ in $L^2(\Omega)$ for some $\bar \lambda\in L^2(\Omega)$.
By definition of $\M$, we find sequences 
$\{\lambda^k_a\},\{\lambda^k_b\}\subset L^2(\Omega)$ satisfying
$\lambda^k=\lambda^k_a-\lambda^k_b$, $\lambda^k_a,\lambda^k_b\leq 0$ almost everywhere on $\Omega$,
and
\begin{equation}\label{eq:box_constraints_bound}
	-\<\lambda^k_b,u_b-u^k\>_{L^2(\Omega)}-\<\lambda^k_a,u^k-u_a\>_{L^2(\Omega)}\leq r^k
\end{equation}
for all $k\in\N$. 
Next, let us set $\tilde\lambda^k_a:=\min(\lambda^k,0)$ and $\tilde\lambda^k_b:=-\max(\lambda^k,0)$
where $\min$ and $\max$ have to be interpreted in a pointwise sense.
Clearly, it holds $\lambda^k=\tilde\lambda^k_a-\tilde\lambda^k_b$ for all $k\in\N$.
By construction, we additionally have $0\geq\tilde\lambda^k_a\geq\lambda^k_a$ as well as
$0\geq\tilde\lambda^k_b\geq\lambda^k_b$ for all $k\in\N$. 
Some rearrangements in \eqref{eq:box_constraints_bound} as well as $u_a-u_b\leq 0$ almost everywhere
on $\Omega$ yield
\begin{equation}\label{eq:box_constraints_bound_norm_minimal_multipliers}
	\begin{aligned}
	r^k
	&\geq
	\<\lambda^k_b-\lambda^k_a,u^k\>_{L^2(\Omega)}-\<\lambda^k_b,u_b\>_{L^2(\Omega)}
		+\<\lambda^k_a,u_a\>_{L^2(\Omega)}\\
	&=
	\<\lambda^k_b-\lambda^k_a,u^k\>_{L^2(\Omega)}+\<\lambda^k_a-\lambda^k_b,u_b\>_{L^2(\Omega)}
		+\<\lambda^k_a,u_a-u_b\>_{L^2(\Omega)}\\
	&\geq
	\<\tilde\lambda^k_b-\tilde\lambda^k_a,u^k\>_{L^2(\Omega)}
		+\<\tilde\lambda^k_a-\tilde\lambda^k_b,u_b\>_{L^2(\Omega)}
		+\<\tilde\lambda^k_a,u_a-u_b\>_{L^2(\Omega)}\\
	&=
	\<\tilde\lambda^k_b-\tilde\lambda^k_a,u^k\>_{L^2(\Omega)}-\<\tilde\lambda^k_b,u_b\>_{L^2(\Omega)}
		+\<\tilde\lambda^k_a,u_a\>_{L^2(\Omega)}\\
	&=
	-\<\tilde\lambda^k_b,u_b-u^k\>_{L^2(\Omega)}-\<\tilde\lambda^k_a,u^k-u_a\>_{L^2(\Omega)}
	\end{aligned}
\end{equation}
for each $k\in\N$. 
As a consequence, we can exploit the sequences $\{\tilde\lambda^k_a\}$ and
$\{\tilde\lambda^k_b\}$ in order to represent $\lambda^k\in\M(u^k,r^k)$.
Due to $\lambda^k\wto\bar\lambda$ in $L^2(\Omega)$, the sequence $\{\lambda^k\}$ is bounded
in $L^2(\Omega)$. Thus, the trivial estimates $\norm{\tilde\lambda^k_a}_{L^2(\Omega)}\leq\norm{\lambda^k}_{L^2(\Omega)}$
and $\norm{\tilde\lambda^k_b}_{L^2(\Omega)}\leq\norm{\lambda^k}_{L^2(\Omega)}$ 
show that $\{\tilde\lambda^k_a\}$ and $\{\tilde\lambda^k_b\}$
are bounded in $L^2(\Omega)$, too.
Thus, these sequences possess weakly convergent subsequences.
We assume w.l.o.g.\ that the convergences $\tilde\lambda^k_a\wto\lambda_a$ and $\tilde\lambda^k_b\wto\lambda_b$
hold true. By weak sequential closedness of $\{v\in L^2(\Omega)\,|\,v\leq 0\text{ a.e.\ on }\Omega\}$,
we have $\lambda_a,\lambda_b\leq 0$ almost everywhere on $\Omega$.
Since it holds $\lambda^k\wto\bar\lambda$ and $\lambda^k=\tilde\lambda^k_a-\tilde\lambda^k_b$, uniqueness
of the weak limit yields $\bar\lambda=\lambda_a-\lambda_b$.
Taking the limit $k\to\infty$ in 
\eqref{eq:box_constraints_bound_norm_minimal_multipliers} implies
\[
	0\geq-\<\lambda_b,u_b-\bar u\>_{L^2(\Omega)}-\<\lambda_a,\bar u-u_a\>_{L^2(\Omega)},
\]
where we used the convergences $u^k\to\bar u$, $\tilde\lambda^k\wto\lambda_a$, and $\tilde\lambda^k_b\wto\lambda_b$
in $L^2(\Omega)$.
Due to $\bar\lambda=\lambda_a-\lambda_b$, we have shown $\bar\lambda\in\M(\bar u,0)$.
Particularly, $\bar u$ is sw-AKKT regular.

\section{Application to Safeguarded Augmented Lagrangian Methods}\label{Sec:Apps}

In this section, we want to show that the safeguarded augmented Lagrangian method (ALM), applied 
to \eqref{Eq:Opt}, generates a w-AKKT sequence under appropriate assumptions, 
and deduct consequences for the convergence behavior 
from this observation. The safeguarded augmented Lagrangian methods have become popular for finite- and 
infinite-dimensional optimization problems, see \cite{Andreani2007,Andreani2008,Andreani2019,Birgin2014} 
for the finite-dimensional perspective and \cite{BoergensKanzow2019local,Kanzow2016b} for the 
infinite-dimensional view. Since augmented Lagrangian methods are at their core Hilbert space methods, 
we presume that the constraint space $ Y $ from \eqref{Eq:Opt} is densely embedded in a Hilbert space $ H $ 
such that we have the Gelfand triple structure 
$ Y \overset{d}{\hookrightarrow} H = H^* \overset{d}{\hookrightarrow} Y^* $. Further we require 
that the constraint is well represented in the Hilbert space, i.e., we assume that there is a 
closed convex set $\mathcal{K}\subset H$ such that $ e^{-1} ( \mathcal{K} ) = K $, where $ e $ represents the 
dense embedding $Y \overset{d}{\hookrightarrow} H $. Thus, problem \eqref{Eq:Opt} is equivalent to 
\begin{equation}\tag{$P_{\SpaceH}$}\label{Eq:OptH}
	\begin{aligned}
		&\minimize_{x\in\ConX}&&f(x)&
		&\subjectto&& e(G(x))\in\ConH.&
	\end{aligned}
\end{equation}
For better readability, the embedding $ e $ will be omitted in the sequel.

We now turn to the multiplier-penalty method for the optimization 
problem \eqref{Eq:Opt}. To this end, we define the augmented Lagrangian of \eqref{Eq:OptH} as follows.

\begin{definition}\label{Dfn:AL}
	For $\rho>0$, the \emph{augmented Lagrange function} or \emph{augmented Lagrangian} of \eqref{Eq:OptH} is the function
	$\Lag_{\rho}\colon X\times\SpaceH\to\R$ defined by
	\begin{equation}\label{Eq:AL}
	\forall x\in X\,\forall\lambda\in\SpaceH\colon\quad
	\Lag_{\rho}(x,\lambda):=f(x)+\frac{\rho}{2} d_{\ConH}^2 \mleft( G(x)+\frac{\lambda}{\rho}
	\mright)-\frac{\norm{\lambda}_{\SpaceH}^2}{2\rho}.
	\end{equation}
\end{definition}
Note that there are other variants of $ \Lag_{\rho} $ in the literature. 
However, these differ from \eqref{Eq:AL} only by an additive constant 
(w.r.t.\ $x$).

For the construction of our algorithm, we will need a means of controlling
the penalty parameter. To this end, we define the utility function
\begin{equation*}
	\forall x\in X\,\forall\lambda\in\SpaceH\,\forall\rho>0\colon\quad
	V(x,\lambda,\rho):=\norm*{G(x)-P_{\ConH}\paren*{G(x)+
	\rho^{-1}\lambda}}_{\SpaceH}.
\end{equation*}
This definition enables us to formulate our algorithm as follows.

\begin{myalgorithm}[ALM for constrained optimization]\label{Alg:ALM}
		Let $(x^0,\lambda^0)\in X\times \SpaceH$, $\rho^0>0$, and a nonempty, bounded set 
		$B\subset \SpaceH$ be given.
		Furthermore, fix $\gamma>1$, $\tau\in(0,1)$, and set $k:=0$.
		\begin{enumerate}[topsep=1ex,parsep=0ex,leftmargin=*,label=\textbf{Step~\arabic*.}]
			\item If $(x^k,\lambda^k)$ satisfies a suitable termination criterion: STOP.
			\item Choose $\BddMul^k\in B$ and compute an approximate solution $x^{k+1}$ of the problem
			\begin{equation}\label{Eq:AL:xSubproblem}
			\minimize_{x\in\ConX}\, \Lag_{\rho^k}(x,\BddMul^k).
			\end{equation}
			\item Update the vector of multipliers to
			\begin{equation*}
			\lambda^{k+1}:=\rho^k \mleft[ G(x^{k+1})+\frac{\BddMul^k}{\rho^k}
			-P_{\ConH}\mleft(G(x^{k+1})+\frac{\BddMul^k}{\rho^k}\mright) \mright].
			\end{equation*}
			\item Let $V^{k+1}:=V(x^{k+1},\BddMul^k,\rho^k)$ and set
			\begin{equation*}
			\rho^{k+1}:=
			\begin{cases}
			\rho^k & \text{if }k=0\text{ or }V^{k+1}\le \tau V^k,\\
			\gamma \rho^k & \text{otherwise}.
			\end{cases}
			\end{equation*}
			\item Set $k\leftarrow k+1$ and go to Step~1.
		\end{enumerate}
\end{myalgorithm}

Note that \Cref{Alg:ALM} differs from the classical augmented Lagrangian method by the
introduction of the bounded sequence $ \{ w^k \} $. The classical method is obtained by
replacing $ w^k $ by $ \lambda^k $ everywhere. Hence, both methods coincide whenever one takes
$ w^k = \lambda^k $ as long as $ \lambda^k $ remains bounded, say $ \lambda^k \in B $ for
the user-specified set $ B $ from \Cref{Alg:ALM}. For an unbounded sequence $ \{ \lambda^k \} $,
however, the global convergence properties of the above (safeguarded) augmented Lagrangian
method are stronger than for the classical ALM, cf.\ the example given in \cite{Kanzow2017}.

Let us stress that the sequence $ \{ \lambda^k \} $ generated by \Cref{Alg:ALM} belongs
to the Hilbert space $ H $, but will usually be viewed as a sequence in the (larger)
space $ Y^* $ since boundedness of this sequence is usually easier to verify in this
dual space than in $ H $ itself (note that, despite the boundedness of $ \{ w^k \} $, the 
sequence $ \{ \lambda^k \} $ might still be unbounded). The reader might therefore wonder
why we introduce the Gelfand structure $\SpaceY\hookrightarrow H\hookrightarrow \SpaceY^*$
with a Hilbert space $H$. The reason for that is twofold.
On the one hand, the augmented Lagrangian method is mainly a Hilbert space technique due to the fact that
we compute projections (and $ Y $ itself might not be a Hilbert space). 
On the other hand, the presence of $H$ gives us some more freedom for the design of the actual method.
If $ \SpaceY $ itself is a Hilbert space, it seems very natural, at a first glance, to take
$ H := Y $. However, if $ Y := H_0^1(\Omega) $ would be taken as the Hilbert space $ H $,
we would have to compute projections w.r.t.\ the norm in $H^1_0(\Omega)$, and these projections
are expensive to calculate. In this case, it is a nearby idea to embed the
Sobolev space $Y$ into $ H:=L^2(\Omega) $, where projections are usually much cheaper to compute.

So far, we have not specified what constitutes an ``approximate solution'' 
in Step 2 of \cref{Alg:ALM}. Clearly, there are multiple possibilities when solving the 
subproblem \eqref{Eq:AL:xSubproblem}; for instance, we could look for global minima
or KKT points. Here, we are only interested in the case where the 
subproblems are solved by computing inexact KKT points. To this end, we state
the following (natural) assumption regarding the quality by which the subproblems
are solved in each inner iteration of \Cref{Alg:ALM}.

\begin{assumption}\label{Asm:xSubAccuracy}
We assume that there is a sequence $\{\varepsilon^k\}\subset X^*$ with $\varepsilon^k\wtos 0$
such that
$x^{k+1}\in\ConX$ and $\varepsilon^{k+1}-(\Lag_{\rho^k})'_x(x^{k+1},\BddMul^k)\in \Nor_{\ConX}(x^{k+1})$ 
hold for all $k\in\N$.
\end{assumption}

The next lemma verifies the first part of the w-AKKT sequence property stated in \Cref{Def:AKKTseq}.
A proof of this result can be found in \cite[Lem.\ 5.2]{Kanzow2018}.

\begin{lemma}\label{Lem:firstPartOfAKKT}
Let $\{(x^k,\lambda^k)\}\subset X\times H$ be a sequence generated by \Cref{Alg:ALM}. 
Then there is a null sequence $\{r^k\}\subset [0,+\infty)$ such that $(\lambda^k,y-G(x^k))_{\SpaceH}\le r^k$ 
holds for all $y\in\ConH$ and $k\in\N$. 
\end{lemma}

By the last lemma and $ e(K) \subset \mathcal{K} $, we also have $
\dual{\lambda^k}{y-G(x^k)}_Y \le r^k$ 
for all $y\in K$ and $k\in\N$. This shows that one of the two defining properties in the
definition of w-AKKT sequences is satisfied for problem \eqref{Eq:Opt} (viewed as an
optimization problem with the primal-dual pair $ (x^k, \lambda^k) $ generated in 
$ X \times Y^* $, whereas the optimization problem \eqref{Eq:OptH} would have to view
this sequence as belonging to the space $ X \times H $). The second part from the
definition of a w-AKKT sequence is a requirement in $ X^* $ and therefore identical for
\eqref{Eq:Opt} and \eqref{Eq:OptH}. A verification of this second condition yields the 
following result.

\begin{theorem}\label{Thm:ALgenAKKTseq}
Suppose that the subproblems in \eqref{Eq:AL:xSubproblem} are solved such that \Cref{Asm:xSubAccuracy} 
holds. Then \Cref{Alg:ALM} generates a w-AKKT sequence $\{(x^k,\lambda^k)\}\subset X\times H$
of \eqref{Eq:OptH} (and, particularly, due to $\{(x^k,\lambda^k)\}\subset X\times Y^*$, of \eqref{Eq:Opt},
too).
\end{theorem}

\begin{proof}
	First, we obtain
	\begin{equation*}
	(\Lag_{\rho^k})_x'(x,\BddMul^k)
	= f'(x)+\rho^k G'(x)^* \mleft[ G(x)+\frac{\BddMul^k}{\rho^k}-P_{\ConH}\mleft(
	G(x)+\frac{\BddMul^k}{\rho^k}\mright) \mright].
	\end{equation*}
	Thus, we deduce from the definition of $ \lambda^{k+1} $ that
	$(\Lag_{\rho^k})'_x(x^{k+1},\BddMul^k)=\Lag'_x(x^{k+1},\lambda^{k+1})$.
	Consequently, \Cref{Asm:xSubAccuracy} yields
	$\varepsilon^{k+1}-\Lag'_x(x^{k+1},\lambda^{k+1}) \in \Nor_{\ConX}(x^{k+1})$ for all $k\in\N$, 
	and we additionally have the convergence $\varepsilon^k\wtos 0$.
	Together with \Cref{Lem:firstPartOfAKKT}, the claim follows.
\end{proof}

Note that \Cref{Alg:ALM} is a kind of penalty method, and therefore suffers from the drawback
that (weak or strong) limit points generated by this method may not be feasible to \eqref{Eq:OptH}. The general
convergence theory for this method presented in \cite{BoergensKanzow2019local} shows, 
however, that we usually get (at least) a KKT point of the program
\[
		\minimize_{x\in\ConX}\;(d_\ConH^2\circ G)(x).
\]
Hence,
these limit points can be interpreted in a suitable way, namely as being stationary points of
the constraint violation. In the following result, we can 
therefore concentrate on the situation where we have a feasible weak limit point of $\{x^k\}$
at hand. 
This is the situation where 
w-AKKT regularity 
can be applied in order to obtain the following convergence
result which puts the corresponding theorem from \cite{BoergensKanzow2019local}, where 
validity of RCQ was assumed, in some other light. 

\begin{corollary}
Suppose that the subproblems in \eqref{Eq:AL:xSubproblem} are solved such that \Cref{Asm:xSubAccuracy} 
holds.
Assume that $ f' $ is weak-to-weak* continuous. 
Suppose that the sequence $\{ x^k \} $ generated by \Cref{Alg:ALM} 
admits a feasible weak accumulation point $ \bar x\in\Feas $ that 
is w-AKKT regular. 
\eqref{Eq:Opt}. Then $\bar x$ is a KKT point of \eqref{Eq:Opt}.
\end{corollary}

\begin{proof}
This follows immediately from \Cref{wAKKTpIsKKT,Thm:ALgenAKKTseq}.
\end{proof}

\section{Final Remarks}\label{Sec:Final}

In this paper, we have shown that certain
types of AKKT regularity 
may serve as constraint
qualifications which apply to abstract optimization problems in Banach spaces. 
We discussed the relation of these new CQs to already existing ones and demonstrated by means
of some examples that there are several practically relevant situations where 
AKKT regularity conditions
apply while Robinson's CQ is generally violated. 
It has been shown that
AKKT regularity
may also be helpful for the convergence
analysis associated with optimization algorithms.
In the future, it remains to be seen to what extent these new constraint qualifications enrich
the theory on optimization in Banach spaces and optimal control.
Particularly, some more situations have to be identified where
AKKT regularity
can be applied in a profitable way.
Let us point out that even in the case where $X$ and $Y$ are finite dimensional, 
there seems to be lots of potential hidden in 
AKKT regularity properties
since our generalized version from \Cref{Def:AKKTCQ} applies to the abstract model
\eqref{Eq:Opt}, and the latter covers, e.g., semidefinite or second-order cone programs.

In the finite-dimensional setting, 
AKKT regularity
has been successfully generalized
to the setting of \emph{mathematical programs with complementarity constraints} (MPCCs), see
\cite{AndreaniHaeserSecchinSilva2019,Ramos2019}.
Recently, the notion of MPCCs has been studied in the rather general context of Banach spaces in
\cite{Mehlitz2017,MehlitzWachsmuth2016,Wachsmuth2015,Wachsmuth2016}. 
Therein, some reasonable problem-tailored notions of stationarity have been defined
and associated CQs have been investigated. 
However, due to the limited number of CQs addressing \eqref{Eq:Opt},
only a few problem-tailored MPCC-CQs have been derived in these papers. 
Motivated by the results from 
\cite{AndreaniHaeserSecchinSilva2019,Ramos2019}, 
we aim to study applicability of the concept of  
AKKT regularity
(or some problem-tailored counterparts)
in the setting of MPCCs in Banach spaces in the future.

One can easily check that MPCCs can be modeled in the form \eqref{Eq:Opt} whenever the sets
$K$ and $C$ are allowed to be non-convex. Thus, the question arises whether one can extend
the theory of this paper to this even more general setting. Observing that our analysis is based
on sequences, this naturally would lead to the involvement of the limiting normal cone.
A brief investigation of the proof of \cite[Theorem~3.3]{Ramos2019} reveals that this should work
fine in the finite-dimensional setting. In the infinite-dimensional situation, it has to
be expected that additional assumptions on the data have to be 
satisfied e.g.\ in order to obtain a counterpart of \cref{Prop:OptAKKT2}
or to apply the calculus rules of limiting variational analysis.
This observation is reflected by the results from \cite[Section~5.1.2]{Mordukhovich2006a}
where KKT-type necessary optimality conditions for optimization problems of type
\eqref{Eq:Opt} with non-convex sets $K$ and $C$ are derived in the presence of suitable
constraint qualifications and additional so-called \emph{sequential
compactness} assumptions on the data $K$ or $C$ which are quite restrictive in the
function space setting, see \cite{Mehlitz2019}. However, these assumptions might be 
interpreted as a price one has to pay for the absence of convexity.
Indeed, the theory in this paper does not require any sequential compactness of the
data.

\subsection*{Acknowledgement}
We thank the two anonymous referees
for the careful reading of the manuscript
and for pointing out \cref{lem:referee}.

\bibliographystyle{siamplain}
\bibliography{AKKTReg}

\end{document}